\documentclass[reqno,a4paper]{amsart}
\usepackage{times}
\usepackage{amsmath,amssymb}
\usepackage{bm, bbm}

\chardef\bslash=`\\ 

\newtheorem{definition}{Definition}
\newtheorem{thm}{Theorem}
\newtheorem{lem}{Lemma}

\newtheorem{cor}{Corollary}
\newtheorem{remark}{Remark}
\newtheorem{remarks}{Remarks}
\newtheorem{prop}{Proposition}

\def\JS{\textrm{JS}}
\def\js{\textrm{js}}

\def\Jc{\textrm{Jc}}

\def\inv{\text{inv}}
\def\Orb{\textrm{Orb}}
\def\sor{\text{Sor}}
\def\rec{\text{rec}}

\def\N{\mathbb{N	}}

\newcommand{\ds}{\displaystyle}

\newcommand{\qbinom}[3]{ {\left[\begin{array}{@{\, }c@{\, } } #1\\  #2 \end{array}\right]_{#3}} }

\newcommand{\qJc}[3]{ \text{Jc}_{#1}^{#2} (#3;q)}
\newcommand{\qJS}[3]{ \text{JS}_{#1}^{#2} (#3;q)}

\newcommand{\eval}[2][\right]{\relax
  \ifx#1\right\relax \left.\fi#2#1\rvert}

\numberwithin{equation}{section}

\begin{document}

\title[$q$-classical polynomials and $q$-Jacobi-Stirling numbers]{$q$-differential equations for $q$-classical polynomials and $q$-Jacobi-Stirling numbers
\footnote{\normalfont version accepted for publication in {\it Mathematische Nachrichten}}}

\author{Ana F. Loureiro}
\address[Ana F. Loureiro]{School of Mathematics, Statistics \& Actuarial Science (SMSAS), Cornwallis Building, University of Kent, Canterbury, Kent CT2 7NF (U.K.)}
\email[Corresponding author]{A.Loureiro@kent.ac.uk}

\author{Jiang Zeng}
\address[J. Zeng]{Universit\'e de Lyon, Universit\'e Claude Bernard Lyon~1, Institut Camille Jordan, UMR 5208 du CNRS,
 43 boulevard du 11 novembre 1918, 
 69622 Villeurbanne cedex (France)}

%

\begin{abstract}
 We introduce, characterise and provide a combinatorial interpretation for the so-called $q$-Jacobi-Stirling numbers. This study is motivated by their key role in the (reciprocal) expansion of any power of a second order $q$-differential operator having the $q$-classical polynomials as eigenfunctions in terms of other even order operators, which we explicitly construct in this work. The results here obtained can be viewed as the $q$-version of those given by Everitt {\it et al.} and by the first author, whilst the combinatorics of this new set of numbers is a $q$-version of the Jacobi-Stirling numbers given by Gelineau and the second author.
\end{abstract}

\maketitle


\section{Introduction and main results}

The present paper  concerns with the explicit construction of even order $q$-differential operators having the so-called $q$-classical polynomials as eigenfunctions. Motivated by this goal, we introduce, develop and give a combinatorial interpretation to the so-called $q$-Jacobi-Stirling numbers. Among other things, 
this work consists of a $q$-version of the studies in \cite{AndrewsEggeGawLittle2013, Egge2010,LittleEvAlJacobi2007,LittleEvAlJacobi2000,ZengGelineau,GLZ,Loureiro2010}, while assembling the information in a coherent framework. To our best knowledge the results here presented are new in the theory.

At the centre is the $q$-derivative operator or  {\it Jackson}-derivative, here denoted as $D_{q}$ and defined for any polynomial $f$ as follows 
\begin{equation}\label{q derivative}
\big( D_{q} f \big)(x) := \frac{f(qx) - f(x)}{(q-1) x} \qquad \text{if}\quad x\neq 0,
\end{equation}
and $(D_{q} f)(x) =f'(0)$ if $x=0$, 
where $q$ is a complex number such that $q\neq 0$ and $|q|\neq 1$.
Throughout  the text, the symbols $j,k,m$ and $n$ will be allocated for integer values and we will write $n\geqslant k$ to mean all the integers $n$ greater than or equal to $k$. As $q\to 1$, the $q$-derivative operator $D_q$ reduces to the derivative operator, $D_{q}\to D$, and its action on the sequence of monomials $\{x^n\}_{n\geqslant 0}$  is given by 
$$
	\left(D_q \zeta^n\right)(x) =\left\{\begin{array}{lcl} 
							[n]_q \ x^{n-1} &\text{if} & n\geqslant 1, \\
							0 &\text{if} &  n=0, 
						\end{array}\right.
$$
where $[a]_{q} := \frac{q^a - 1}{q-1}$ for any complex number $a$. 

Recalling the classical $q$-notations found on \cite{GasperRahman,Koekoek},  the {\it $q$-shifted factorial} (also known as the  {\it$q$-Pochhammer symbol}) is defined by 
 $(x;q)_{k}:=\prod\limits_{i=0}^{k-1}(1-xq^i)$, for $k\geqslant 1$, and $(x;q)_{0}:=1$, whilst the $q$-binomial coefficients are defined by 
 $$
{n\brack k}_q=\frac{(q;q)_n}{(q;q)_k(q;q)_{n-k}}\qquad \text{if}\quad 0\leqslant
 k\leqslant n,
$$
and ${n\brack k}_q=0$ otherwise. In addition, for practical reasons, we introduce the supplementary  notation 
$$
 [a;q]_{n}:=\prod_{k=0}^{n-1}[a+k]_q\quad\text{and}\quad  [n]_q!:= [1;q]_{n}.
 $$
 Hence, $[a;q]_{n}:=(1-q)^{-n} (q^a;q)_n$, for any $n\geqslant0$.

Throughout the text, we consider sequences of polynomials of successive degrees starting at zero and leading coefficient equals to one, which we refer to as {\it monic polynomial sequences} (MPS in short). An MPS forms a basis of the vector space of polynomials with complex coefficients,  denoted by $\mathcal{P}$, and we will mainly deal with the so-called $q$-classical polynomial sequences. These are monic orthogonal polynomial sequences (soon abbreviated to MOPS) possessing Hahn's property \cite{Hahn1949}, that goes as follows: 
\begin{definition}\label{def q cla} An MOPS $\{ P_n\}_{n\geqslant 0}$ is called {\it $q$-classical} whenever the sequence of its $q$-derivatives,  $\{ P_n^{[1]}\}_{n\geqslant 0}$, where $P_n^{[1]}(x):=\frac{1}{[n+1]_q} \left(D_q P_{n+1}\right)(x)$, is also an MOPS. 
\end{definition}
Relevant properties about orthogonality and the $q$-classical character of a sequence are summarised in \S \ref{sec: preliminaries}, but we highlight the fact that an MOPS $\{ P_n\}_{n\geqslant 0}$ is $q$-classical if and only if there exist two polynomials $\Phi$ (of degree at most 2) and $\Psi$ (of degree 1)  such that 
\begin{equation}\label{Lq def}
	\mathcal{L}_{q}[P_n](x):=\Big(   \Phi(x)  D_{q}\circ D_{q^{-1}}  - \Psi(x)   D_{q^{-1} }  \Big) P_n(x) = \chi_{n} P_n(x) \ , \ n\geqslant 0,
\end{equation} 
where 
\begin{equation}\label{eigenval Bohner chi n}
\begin{array}{lcl}
	\chi_{n}
	&=& \left\{ \begin{array}{lcll}  
	- [n]_{q^{-1}} \Psi'(0) 
		& \text{if} & \deg\Phi=0,1, \vspace{0.2cm}\\
	 ~ [n]_{q^{-1}}  \left(  [n-1]_{q} 
		-    \Psi'(0) \right) 
		& \text{if} & \deg\Phi=2 , &   n\geqslant 0.
	\end{array}\right. 
	\end{array}
\end{equation}
In other words, the elements of a $q$-classical MOPS are eigenfunctions of the second-order $q$-differential operator $\mathcal{L}_q$ of $q$-Sturm-Liouville type (see Proposition \ref{Prop qclassical} for more details). 

Regarding their importance in various fields of mathematics, there is a considerable bibliography on the subject -- without any attempt for completion, we refer to \cite{Bateman, GasperRahman, Hahn49Hyper,IsmailBook, KherijiMar2002, Koekoek, Lesky94, Lesky01, Nikiforov93}. In particular, they appear in the literature  in the framework of discretisations of hypergeometric second order differential operators in $q$-lattices (see \cite{Koekoek,Nikiforov91}). Structural properties of such polynomials have been intensively studied in \cite{AlvMed01,KherijiMar2002,AlvMedMarc01} following the program developed for P. Maroni for the classical ones by using an algebraic approach. From an historical point of view, in \cite{Hahn1949}, q-classical orthogonal polynomial sequences appear for the first time as those sequences of orthogonal polynomials $\{p_n\}_{n\geqslant 0}$ such that the sequence of polynomials $\{q_n(x):= D_q p_{n+1}(x)\}_{n\geqslant 0}$ is also orthogonal. The $q$-classical polynomials include the q-polynomials of Al-Salam and Carlitz, the Stieltjes-Wigert, the  q-Laguerre, the Little q-Jacobi, and an exhaustive list can be found in Table \ref{Table qclassical} -- see \cite{KherijiMar2002} for a detailed survey.

When $q$ tends to $1$, we recover the classical polynomials of Hermite, Laguerre, Bessel or Jacobi. These sequences are the only orthogonal polynomial sequences whose elements are eigenfunctions of  the second-order linear differential operator $\mathcal{L}:=\mathcal{L}_{1}$, known as the Bochner differential operator, named after \cite{Bochner1929}. 
Krall \cite{Krall1938} showed that if the elements of a classical polynomial sequence are eigenfunctions of a linear differential operator, then it must be of even order. In fact, as observed in  \cite[Theorem 1]{Miranian}, there exists a polynomial $P$ so that any even order differential operator having  classical polynomials as eigenfunctions can be written as $P(\mathcal{L})$. Thus, if $\mathcal{O} = \sum_{k=0}^{2n} a_k(x)D^k$, where $a_k(x)$ are polynomials of degree at most $k$, then a natural question is to state the connection between the polynomials $a_k(x)$ and the polynomial coefficients $\Phi$ and $\Psi$ in the differential operator $\mathcal{L}$. This  question has been addressed  in  \cite{Loureiro2010} and partly  in \cite{LittleEvAlJacobi2007,LittleEvAlJacobi2000}. More precisely, explicit expressions for any integral power of $\mathcal{L}$ were obtained for each of the classical polynomial families, orthogonal with respect to a positive-definite measure in a series of papers by Everitt {\it et al.} (see \cite{LittleEvAlJacobi2007,LittleEvAlJacobi2000} and the references therein).  Meanwhile, in \cite{Loureiro2010} a similar study was carried out in a more coherent framework gathering the properties of all the classical families (not restricting to the positive-definite case, and therefore including, for example, the Bessel case). The key to attain such expressions were the Stirling numbers in the study of the cases of Hermite and Laguerre polynomials, whereas the cases of Bessel and Jacobi polynomials required the introduction of a new pair of numbers: the so-called  {$\alpha$-modified Stirling numbers} $\widehat{s}_\alpha(n,k)$ and $\widehat{S}_\alpha(n,k)$. These are best known as Jacobi-Stirling numbers (as called in \cite{LittleEvAlJacobi2007}), which can be defined as follows
$$ 
\prod_{i=0}^{n-1}(x-i(\alpha+i))=\sum_{k=0}^{n} \widehat{s}_\alpha(n,k)x^{k} 
\quad\text{and} \quad x^{n}=\sum_{k=0}^{n}\widehat{S}_\alpha (n,k)\prod_{i=0}^{i-1}(x-i(\alpha+i)), \ n\geqslant 0, 
$$ 
where $\alpha$ represents a complex number. Since their introduction, several properties of the Jacobi-Stirling numbers  and its companions including combinatorial interpretations have been established, see \cite{AndrewsEggeGawLittle2013, Egge2010,ZengGelineau, GLZ, LinZen14, Mongelli12}. 

In this paper, we obtain explicit expressions for any positive integer power of the $q$-differential operator $\mathcal{L}_q$ given in \eqref{Lq def} and the key to achieve this is via the introduction of a new set of numbers, that we call as $q$-Jacobi-Stirling numbers, to which we obtain several properties as well as a combinatorial interpretation. The results here obtained are the $q$-version to those in \cite{Loureiro2010} (and also in \cite{LittleEvAlJacobi2007,LittleEvAlJacobi2000}), as well as to those in \cite{AndrewsEggeGawLittle2013, Egge2010,ZengGelineau, GLZ,Miranian,Mongelli12}, since we provide here combinatorial interpretations to the arisen coefficients and eigenvalues. This study has the merit of addressing all the $q$-classical polynomial sequences as whole in a coherent framework that brings together generalisations of the $q$-differential equation \eqref{Lq def} and associated combinatorial interpretations. 

Our starting point is  Theorem \ref{Thm: 2k th order diff eq}, which 
gives an explicit expression of $q$-differential operators of arbitrary even order, $\mathcal{L}_{k;q}$ say, having the $q$-classical polynomials as eigenfunctions. Obviously, this sequence of operators have $\mathcal{L}_q$  as a particular (and simplest) case.  

\begin{thm}\label{Thm: 2k th order diff eq} Let $\{P_{n}\}_{n\geqslant 0}$ be a ${q}$-classical MOPS fulfilling \eqref{Lq def}. For any  integer $k\geqslant 0$,  the polynomial $y(x)=P_n(x)$  also  fulfils  
\begin{align}\label{2k order diff eq}
	\mathcal{L}_{k;q}[y](x) 
	:= \sum_{\nu=0}^{k} \Lambda_{k,\nu}(x;q) 
		\bigg(D_{q^{-1}}^{k-\nu} \circ D_{q}^{k}\ y \bigg)(q^{-\nu}x)  
	\ =  \ \Xi_{n}(k;q)\  y(x), 
\end{align}
where, for each $ \nu=0,1,\ldots, k$, 
$$
	\Lambda_{k,\nu}(x;q)  
	=\frac{q^{-(k-\nu)(\nu+1)}}{[k-\nu]_q!}  \ 
		\Big( \prod_{\sigma=1}^\nu 
	 \chi_{\sigma}^{[ k-\sigma]} 	  \Big) 
	 \ \Big(\prod_{\sigma=0}^{k-\nu-1}q^{-\sigma\deg\Phi } \Phi(q^{\sigma}x)\Big)
		D_{q}^{k-\nu} P_{k}(x) , 
$$
with $\quad 
	\chi_{n}^{[k]} 
	= q^{ -k\deg \Phi -1} [n]_{q^{-1}}  \Big(  [n+2k]_{q}  \frac{\Phi''(0)}{2} + z \Big) 
$,  $ \ z= -(\frac{\Phi''(0)}{2}  +q\Psi'(0)) $ \  and  
\begin{equation}\label{Xi n k expression}
\Xi_{n}(k;q) =  \prod_{\sigma=0}^{k-1}  \chi_{n-\sigma}^{[\sigma]}
= \left\{\begin{array}{@{}l@{\ }l@{}}
	 \displaystyle \left(- q^{\frac{(k-1)(1- \deg \Phi)}{2}} \  \Psi'(0) \right)^k   
				\ \prod_{\sigma=0}^{k-1} \big( [n]_{q^{-1}} -[\sigma]_{q^{-1}} \big) 
	 	&\text{if }  \deg \Phi =0,1, \\
	\displaystyle q^{-\frac{k(k+1)}{2}}  \prod_{\sigma=0}^{k-1}   \Big(
					  [n]_{q^{-1}} \big( z + [n]_q  \big) 
					- [\sigma]_{q^{-1}} \big( z + [\sigma]_q  \big)    \Big)	
		&\text{if } \deg \Phi =2. \quad 
	 \end{array}\right.
\end{equation}
\end{thm}

Observe that by taking $q\to 1$ in Theorem \ref{Thm: 2k th order diff eq}, we recover the results obtained in \cite[Theorems 2.1 and 2.2]{Loureiro2010}.

In some cases, the orthogonality measures associated to certain $q$-classical sequences $\{P_{n}\}_{n\geqslant 0}$ can be expressed via a weight function, so that there exists a function $W_q(x)$ possessing all the necessary properties with support in $(a,b)$, with $a<b$, possibly infinite, so that
$$
	\int_a^b P_n(x) P_m(x) W_q(x) dx = N_n \delta_{n,m}, \ \text{ with } \ N_n\neq0  , \ n,m\geqslant 0,  
$$
where the symbol $\delta_{n,m}$ stands for the {\it Kronecker delta symbol}.  If a $q$-classical sequence $\{P_{n}\}_{n\geqslant 0}$, orthogonal with respect to the weight function $W_q$, fulfils the $q$-differential equation \eqref{Lq def}, then $W_q(x)$ satisfies 
\begin{equation}\label{eq Uq}
	q^{-1} D_{q^{-1}} \left( \Phi(x) W_q (x)\right) + \Psi(x) W_q (x) = \lambda \ g(x), 
\end{equation}
where $\lambda$ is an arbitrary complex number and $g\neq 0$ is a locally integrable function with rapid decay representing the null form, as it is, for instance, the case of the Stieltjes function 
$
	g(x) =\exp(-x^{1/4}) \sin(x^{1/4}) \mathbbm{1}_{[0,+\infty]}(x)
$ where $\mathbbm{1}_{A}(x)$ stands for the characteristic function of the set $A$: it equals $1$ if $x\in A$ and $0$ otherwise. (For more details we refer to \cite[p.81-82]{KherijiMar2002}.)
If this is the case, a $q$-self-adjoint version of  Theorem \ref{Thm: 2k th order diff eq} is obtained:  
\begin{cor}\label{Cor: Lkq via Uq} With the same assumptions as those in Theorem \ref{Thm: 2k th order diff eq}, if the $q$-classical sequence $\{P_{n}\}_{n\geqslant 0}$ is orthogonal with respect to the weight function $W_q(x)$,  then   the operator in \eqref{2k order diff eq} can be written as 
\begin{equation}\label{Lk via Uq exp1}
	\mathcal{L}_{k;q}[y](x) := 
		q^{-k}\big(W_{q}(x) \big)^{-1} \ 
		D_{q^{-1}}^{k}\left(
		\  \left(\prod_{\sigma=0}^{k-1} q^{-\sigma\deg\Phi} \Phi(q^{\sigma} x) \right) \ W_{q}(x) \bigl(D_{q}^{k} y(x)\bigr)   \right), 
\end{equation}
for any polynomial $y$, and has $P_n(x)$ as eigenfunctions.
\end{cor}

Further to this, we link integral powers of the operator $\mathcal{L}_{q}$ in \eqref{Lq def} with $\mathcal{L}_{k;q}$, via a set of numbers that we introduce and characterise, the so-called $q$-Jacobi-Stirling numbers (see Section \ref{sec: qJacStir}), together with the $q$-Stirling numbers. The characterisation of the $q$-Jacobi-Stirling numbers in Section \ref{sec: qJacStir} includes, among other things, a combinatorial interpretation. The reason for this lies on the fact that the bridge between powers of the operator $\mathcal{L}_{q}$ with $\mathcal{L}_{k;q}$ is tied to the relation between their corresponding eigenvalues $(\chi_n)^k$ and $\Xi_n(k;q)$, respectively. Observe that  $\chi_n$ are linear in $n$ and $\Xi_n(k;q)$ are essentially $q$-shifted factorials when  $\deg \Phi$ equals $0$ or $1$. Hence, the pair of $q$-Stirling numbers of first and second kind $( {c}_{q} (k,j), {S}_{q} (k,j)))_{j,k\geqslant0}$ allows us to bridge them. However, the eigenvalues $\chi_n$ become quadratic in $n$ and $\Xi_n(k;q)$ a product of quadratic polynomials in $n$ when  $\deg \Phi =2$. In order to link powers of $\chi_n$  to $\Xi_n(k;q)$ we introduce a new pair of numbers -- the \emph{$q$-Jacobi-Stirling numbers} of first and second kind, denoted as $( \Jc_{k}^j (z;q^{-1}), \JS_{k}^{j}(z;q^{-1}))_{j,k\geqslant0}$, both depending on a complex valued parameter $z$. With this we are able to show that: 

\begin{thm}\label{Prop: Lk} Let $k$ be a positive integer and $\mathcal{L}_q$ the linear differential operator defined in \eqref{Lq def}. The identities    
\begin{equation} \label{Lq k and Lqk Rel1}
 		\mathcal{L}_{q} ^k[f](x)
	= \left\{\begin{array}{lcl}
	\ds 		 \sum_{j=0}^{k}  {S}_{q^{-1}} (k,j) \,  q^{(\deg\Phi -1)\frac{j(j-1)}{2}} \, (-\Psi'(0))^{k-j}\,  	
				\mathcal{L}_{j;q}[f](x)
			& \text{if} & \deg\Phi=0,1, \\
	\ds 		   \sum_{j=0}^{k}  \JS_{k}^{j}(z;q^{-1}) \,  q^{\frac{j(j+1)}{2}-k}\,  \mathcal{L}_{j;q}[f](x)
			& \text{if} & \deg\Phi=2,
	\end{array}\right.
\end{equation} 
are valid for any polynomial $f$, where $ \mathcal{L}_{k;q}$ is the operator given in \eqref{2k order diff eq} and $z=-(1+q \Psi'(0))$. 
\end{thm} 

Here $\mathcal{L}_{q} ^k[f](x):=\mathcal{L}_{q}\left[ \mathcal{L}_{q}^{k-1}[f] \right](x)$ for $k=1,2,\ldots,$ with the convention $\mathcal{L}_{q} ^0[f](x):=f(x)$. 
The identities \eqref{Lq k and Lqk Rel1} can be reversed to obtain the following pair of reciprocal relations. 

\begin{cor}\label{Cor: reciprocal identities} For any integer $k$ and any polynomial $f$ we have 
\begin{equation}
		\mathcal{L}_{k;q}[f](x)
	= \left\{\begin{array}{lcl}
	\ds 		q^{(1-\deg\Phi)\frac{k(k-1)}{2}}   \sum_{j=0}^{k}  {c}_{q^{-1}} (k,j)  \, (-\Psi'(0))^{k-j}\,  
			\mathcal{L}_{q} ^j[f](x)
			& \text{if} & \deg\Phi=0,1, \\
	\ds 		   \sum_{j=0}^{k}   \Jc_{k}^j (z;q^{-1}) \,  q^{j-\frac{k(k+1)}{2}}\,  
				\mathcal{L}_{q} ^j[f](x)
			& \text{if} & \deg\Phi=2.
	\end{array}\right.
\end{equation}
\end{cor}

 The outline of the manuscript reads as follows. We begin with the characterisation of the \emph{$q$-Jacobi-Stirling numbers} in Section \ref{sec: qJacStir}, providing their properties in \S \ref{subsec: properties qJS}, which are  also set in comparison to those of  the $q$-Stirling numbers and of the Jacobi-Stirling. For this reason, the proofs are either brief or omitted. A combinatorial interpretation is given in \S\ref{subsec: combinatorial} and, at last, \S\ref{subsec: generalisation} is devoted to symmetric generalisations on this new pair of numbers. In Section \ref{Sec: qDiff eq} we revise the core properties of the $q$-classical polynomials and we prove Theorem \ref{Thm: 2k th order diff eq}, Corollary \ref{Cor: Lkq via Uq}, Theorem  \ref{Prop: Lk} and Corollary \ref{Cor: reciprocal identities} given above. Finally in Section \ref{Sec:examples}, we illustrate these results with a detailed analysis of the cases of the Al-Salam-Carlitz, the Stieltjes-Wigert and the Little $q$-Jacobi polynomials.

\section{The $q$-Jacobi-Stirling numbers}\label{sec: qJacStir}

We address the problem of relating powers of a (complex) number $x$ and a corresponding $q$-factorial, and therefore the introduction of $q$-analogs of the Stirling numbers along with their modifications (also known as Jacobi-Stirling numbers).  

The sequence of factorials $
	\left\{  \prod\limits_{i=0}^{n-1} \left(x- [i]_q\right)  \right\}_{n\geqslant 0}
$, 
alike the  sequence of monomials $\{x^n \}_{n\geqslant 0}$, forms a basis  of the vector space of polynomials, here denoted as $\mathcal{P}$. Thus, there is a pair of sequences of numbers bridging each of these $q$-factorials of a complex number $x$ and its corresponding powers -- these are the so-called $q$-Stirling numbers, here denoted as $(c_q(n,k),S_q(n,k))_{0\leqslant k\leqslant n}$. Precisely, for each non-negative integer $n$, the pair $(c_q(n,k),S_q(n,k))_{0\leqslant k\leqslant n}$ (see \cite{Gould}) is such that
\begin{align} 
	 \prod_{i=0}^{n-1} \left(x- [i]_q\right)  = \sum_{k=0}^n (-1)^{n-k}c_q(n,k) x^{k},\label{q fact to monomials1}\\
	x^{n}  = \sum_{k=0}^n S_q(n,k)   \prod_{i=0}^{k-1} \left(x- [i]_q\right) . \label{q fact to monomials2}
\end{align}
A recurrence relation fulfilled by this pair of $q$-Stirling numbers can be readily deduced from the  identity $x \prod\limits_{i=0}^{n-1} \left(x- [i]_q\right) =\prod\limits_{i=0}^{n} \left(x- [i]_q\right) + [n]_q \prod\limits_{i=0}^{n-1} \left(x- [i]_q\right)$, which goes as follows: 
\begin{align} 
 &	c_q(n+1,k+1) = c_q(n,k)  + [n]_q  c_q(n,k+1)  , \label{q Stirling 1st}\\
 &	S_q(n+1,k+1) = S_q(n,k)  + [k+1]_q  S_q(n,k+1) , \label{q Stirling 2nd}
\end{align}
for $0\leqslant k \leqslant n$ with 
$c_q(n,k) =  S_q(n,k) = 0 ,\ \text{ if } k \notin \{1,\ldots, n\}, 
	\ \text{ and } \ c_q(0,0) =  S_q(0,0) =1. 
$
Naturally, the standard pair of Stirling numbers is recovered from $(c_q(n,k),S_q(n,k))_{0\leqslant k\leqslant n}$ when  $q\to1$. 

As it will soon become clear the connection between the operators $\mathcal{L}_{k;q}$ and $\mathcal{L}_q^k$ will be transferred to the problem of establishing inverse relations between the following sequence of factorials, depending on the prescribed parameter $z$ (possibly complex), 
\begin{equation}\label{set mod factorials}
	\left\{\prod_{i=0}^{n-1} \left(x- [i]_q \big(z+[i]_{q^{-1}} \big)\right)  \right\}_{n\geqslant 0}
\end{equation}
and the sequence of monomials. Straightforwardly, the set \eqref{set mod factorials} is a basis of $\mathcal{P}$. For this reason, there exists a pair of coefficients,  $(\qJc{n}{k}{z},\qJS{n}{k}{z})_{0\leqslant k\leqslant n}$ say, performing the change of basis to $\{z^n\}_{n\geqslant 0}$. We will refer to this pair of coefficients as the 
\emph{$q$-Jacobi-Stirling numbers} defined through 
 \begin{gather} 	
\prod_{i=0}^{n-1} \left(x- [i]_q \big(z+[i]_{q^{-1}} \big)\right) = \sum_{k=0}^n (-1)^{n-k} \qJc{n}{k}{z} \, x^{k}, \ n\geqslant0, 
\label{qjs1} \\
	x^{n}  = \sum_{k=0}^n  \qJS{n}{k}{z} \prod_{i=0}^{k-1} \left(x- [i]_q \big(z+[i]_{q^{-1}}\big)\right),\ n\geqslant0.
	\label{qjs2} 
\end{gather}

\begin{remark}
A $y$-version of the  $q$-Stirling numbers was introduced in 
\cite[Definition~3]{KasStanZeng} in order to compute  the moments of  some rescaled Al-Salam-Chihara polynomials.  It is easy to see that the latter generalized Stirling numbers are a rescaled version of the 
Jacobi-Stirling numbers. For example, the $y$-version of the $q$-Stirling numbers of the second kind $S_q(n,k,y)$ are defined by
\begin{align}\label{y-version}
x^n=\sum_{k=0}^n S_q(n,k,y)\prod_{j=0}^{k-1}(x-[j]_q(1-yq^{-j})).
\end{align}
It follows from \eqref{qjs2} and \eqref{y-version}  that $\JS_n^k(z;q)=S_q(n,k,y^{-1})y^{n-k}(1-q^{-1})^k$ with $z=\frac{y-1}{1-q^{-1}}$.
\end{remark}

Bearing in mind the identity fulfilled by the factorials in \eqref{set mod factorials}
\begin{eqnarray*}
	x \ \prod_{i=0}^{k-1} \left(x- [i]_q \big(z+ [i]_{q^{-1}} \big)\right) 
		&=& \prod_{i=0}^{k} \left(x- [i]_q \big(z + [i]_{q^{-1}} \big)\right) \\
		&& + [k]_q\Big(z + [k]_{q^{-1}}\Big) 
			\prod_{i=0}^{k-1} \left(x- [i]_q \big(z+ [i]_{q^{-1}} \big)\right), 
\end{eqnarray*}
we readily deduce the following triangular recurrence relations fulfilled by these two sets of numbers 
\begin{align}
 &	\qJc{n+1}{k+1}{z} =  \qJc{n}{k}{z}
 			+ [n]_q \Big(z+[n]_{q^{-1}} \Big) \qJc{n}{k+1}{z} , 
 				\label{mod q Stirling 1st}\\
 &	\qJS{n+1}{k+1}{z} = \qJS{n}{k}{z})  
 			+ [k+1]_q  \Big(z+[k+1]_{q^{-1}} \Big) \qJS{n}{k+1}{z} ,  			\label{mod q Stirling 2nd}
\end{align}
for $0\leqslant k \leqslant n$, with 
$\qJc{n}{k}{z} =  \qJS{n}{k}{z}= 0 ,\ \text{ if } k \notin \{1,\ldots, n\},$ and $ \qJc{0}{0}{z} =   \qJS{0}{0}{z} =1$.

The pair \eqref{qjs1}-\eqref{qjs2} represents a pair of inverse relations and we have 
$$
	\sum_{j\geqslant 0} \qJc{n}{j}{z} \qJS{j}{k}{z}
	=	\sum_{j\geqslant 0} \qJS{n}{j}{z}\qJc{j}{k}{z}= \delta_{n,k}, \ n,k\geqslant 0,
$$
as it can be derived directly from their definition. 
\subsection{Properties of the $q$-Jacobi-Stirling numbers}\label{subsec: properties qJS}

An explicit expression for the $q$-Jacobi-Stirling numbers of second kind $\qJc{n}{k}{z}$ can be obtained directly from the Newton interpolation formula for a polynomial $f$ that we next recall. 
\begin{lem}[Newton's interpolation formula] Let $b_0, b_1, \ldots, b_{n-1}$ be distinct numbers. 
Then, for any polynomial $f$ of degree less than or equal to $n$ we have 
\begin{align}\label{newton}
f(x)=\sum_{j=0}^n \left( \sum_{r=0}^j\frac{f(b_r)}{\prod_{0\leqslant k\leqslant j, \;k\neq r} (b_r-b_k)} \right) \prod_{i=0}^{j-1} (x-b_i). 
\end{align}
\end{lem}
From \eqref{newton}, with $f(x)=x^n$ and $b_{k}=[k]_q([k]_{q^{-1}}+z)$, and \eqref{qjs2} we derive immediately the following result: 
\begin{prop} For $0\leqslant j\leqslant n$ we have 
\begin{align}\label{qJS2F}
	\qJS{n}{j}{z}=\sum_{r=0}^j (-1)^{j-r} 
	\frac{q^{-{r\choose 2}-r(j-r)}\bigl([r]_q([r]_{q^{-1}}+z)\bigr)^n}
	{[r]_q![j-r]_q!\prod\limits_{0\leqslant k\leqslant j,\; k\neq r}(z+[k+r]_{q^{-1}})}.
\end{align}
\end{prop} 

\begin{remark}
By letting $q\to 1$, the pair $(\qJc{n}{k}{z},\qJS{n}{k}{z})_{0\leqslant k\leqslant n}$ reduces to the Jacobi-Stirling numbers  in \cite{LittleEvAlJacobi2007,ZengGelineau,Mongelli12} and the relation \eqref{qJS2F}  reduces to \cite[(4.4)]{LittleEvAlJacobi2007}.
\end{remark}

For each integer $k\geqslant 1$, it follows from the recurrence relation \eqref{mod q Stirling 2nd} that 
$$
	\sum_{n\geqslant k}\JS_n^{k}(z;q)x^n
	= \frac{x}{1-[k]_q ([k]_{q^{-1}}+z)} \sum_{n\geqslant k-1}\JS_n^{k-1}(z;q)x^n
$$
and therefore 
\begin{align*}
&& \sum_{n\geqslant k}\qJS{n}{k}{z} x^n = \prod_{i=1}^k\frac{x}{1-[i]_q([i]_{q^{-1}}+z)x}.
\end{align*}

%

\begin{thm} \label{thm1} Let $n,k$ be positive integers with $n \geqslant k$.
The Jacobi-Stirling numbers $\JS_n^k(z,q)$ and $\Jc_{n}^{k}(z,q)$
are  polynomials in $z$ of degree $n-k$ with  coefficients in $\N[q,q^{-1}]$.
Moreover, if
\begin{align}\label{eq:defa}
\JS_{n}^{k}(z;q)&=a_{n,k}^{(0)}(q)+a_{n,k}^{(1)}(q)z+\cdots +a_{n,k}^{(n-k)}(q)z^{n-k},\\
\Jc_{n}^{k}(z;q)&=b_{n,k}^{(0)}(q)+b_{n,k}^{(1)}(q)z+\cdots +b_{n,k}^{(n-k)}(q)z^{n-k},\label{eq:defb}
\end{align}
then
$$
 a_{n,k}^{(n-k)}=S_q(n,k),\quad 
b_{n,k}^{(n-k)}=c_q(n,k). 
$$
\end{thm}
\begin{proof}
This follows from \eqref{mod q Stirling 1st} and \eqref{mod q Stirling 2nd} by induction on $n$.
\end{proof}

\begin{remarks}
\begin{enumerate}
\item The leading coefficient of $\JS_{n}^{k}(z;q)$ and $\Jc_{n}^{k}(z;q)$ in \eqref{eq:defa} are the $q$-Stirling numbers of second kind $S_q(n,k)$ and the $q$-Stirling numbers of first kind $c_q(n,k)$, respectively. Hence, we have  
$\lim\limits_{z\to \infty} \frac{\JS_{n}^{k}(z;q)}{z^{n-k}} = S_q(n,k)$ whilst $\lim\limits_{z\to \infty} \frac{\Jc_{n}^{k}(z;q)}{z^{n-k}} = c_q(n,k)$. 

\item Observe that $x- [i]_{q^{-1}} \left( \frac{1}{q-1} + [i]_q\right) = \frac{q}{q-1} \left( \frac{q-1}{q} x - [i]_q\right)$, for $i=0,1,2,\ldots$. Therefore, from the pair of identities \eqref{q fact to monomials1}--\eqref{q fact to monomials2} and \eqref{qjs1}--\eqref{qjs2}, it follows that 
\begin{equation}\label{JSc to Sc}
	\JS_{n}^{k}(\tfrac{1}{q-1};q^{-1}) = \left(\frac{q}{q-1}\right)^{n-k} S_q(n,k)
	\quad \text{ and } \quad 
	\Jc_{n}^{k}(\tfrac{1}{q-1};q^{-1}) = \left(\frac{q}{q-1}\right)^{n-k} c_q(n,k), 
\end{equation}
for $k=0,1,\ldots, n$.

\item When $z=0$, by analogy with the ordinary central factorial numbers (cf. \cite{ZengGelineau}), we can define 
 the corresponding $q$-Jacobi-Stirling numbers as  $q$-central factorial numbers of even indices~:
 \begin{align}\label{central factorials}
U_q({n}, {k}) =\JS_{n}^{k}(0;q),\quad 
V_q({n},{k})=\Jc_{n}^{k}(0;q).
\end{align}
Therefore, the following recurrences hold true~:
\begin{align}
 &	U_q({n+1}, {k+1}) =  U_q(n,k)+[k+1]_q[k+1]_{q^{-1}}U_q(n,k+1),   \ 0\leqslant k \leqslant n,   \label{mod q CF1st}\\
 &	V_q({n+1}, {k+1})= V_q({n},{k})  + [n]_q  [n]_{q^{-1}} V_q({n},{k+1}) , \ 0\leqslant k \leqslant n, 				\label{mod q CF 2nd}
\end{align}
with 
$U_q({n},{k}) = V_q({n},{k})= 0 ,\ \text{ if } k \notin \{1,\ldots, n\},$ and $U_q({0},{0}) = V_q({0},{0})=1$, $n\geqslant 0$.

\end{enumerate}
\end{remarks}
Some value of the $q$-Jacobi-Stirling numbers of first kind are as follows:
\begin{align*}
\Jc_n^{1}(z;q)&= \prod_{k=1}^{n-1} [k]_q(z+[k]_{q^{-1}}),\quad Jc_n^{(n)}(z,q)=1,\\
\Jc_3^{2}(z;q)&=(3+q+q^{-1})+(2+q)z,\\
\Jc_4^{2}(z;q)&=({q}^{-3}+5\,{q}^{-2}+11\,{q}^{-1}+{q}^{3}+11\,q+5\,{q}^{2}+15)\\
&+
 \left( 2\,{q}^{3}+14\,q+8\,{q}^{2}+2\,{q}^{-2}+7\,{q}^{-1}+15
 \right) z+ \left( 4\,q+{q}^{3}+3+3\,{q}^{2} \right) {z}^{2},\\
\Jc_4^{3}(z;q)&=\bigl(3q^{-1}+6+q^{2}+3q+q^{-2}\bigr)+(3+2q+q^2)z.
\end{align*}
Some  values of the $q$-Jacobi-Stirling numbers of second kind are as follows:
\begin{align*}
\JS_n^{1}(z;q)&=(1+z)^{n-1},\quad JS_n^{(n)}(z;q)=1,\\
\JS_3^{2}(z;q)&=(3+q+q^{-1})+(2+q)z,\\
\JS_4^{2}(z;q)&=(9+q^{-2}+q^{2}+5q+5q^{-1})+(11+3q^{-1}+2q^2+8q)z+(3q+3+q^2)z^2.
\end{align*}

\subsection{Combinatorial interpretation of the $q$-Jacobi-Stirling numbers}\label{subsec: combinatorial}

For any positive integer $n$, we consider the set of two copies of the integers:
\begin{align*}
[n]_2=\{1_1, 1_2,\ldots, n_1, n_2\}.
\end{align*}
\begin{definition}
A  Jacobi-Stirling  $k$-partition of $[n]_2$ is a partition of $[n]_2$ into 
  $k+1$  subsets $B_0, B_1,\ldots B_k$ of $[n]_2$ satisfying   the following conditions:
\begin{enumerate}
\item there is a zero block $B_0$, which may be empty and  cannot contain both copies of any $ i \in [n]$,
\item $\forall j\in [k]$, each nonzero block  $B_j$ is not empty and contains the two copies of its smallest element and does not contain both copies of any other number.
\end{enumerate}
\end{definition}
We shall denote the zero block by $\{\ldots\}_0$. 
For example, the partition 
$$\pi=\{\{  2_2, 5_1\}_0 , \{1_1, 1_2 , 2_1 \} , \{3_1, 3_2, 4_2 \}, \{ 4_1, 5_2 \}\}$$
 is not a Jacobi-Stirling 3-partition of $[5]_2$, while 
 $$\pi'=\{\{  2_2, 5_1\}_0 , \{ 1_1, 1_2 , 2_1 \} , \{3_1, 3_2 \}, \{4_1, 4_2, 5_2 \}\}$$
  is a Jacobi-Stirling  3-partition of $[5]_2$.
We order the blocks of a partition
 in increasing order of their minimal elements. By convention, the zero block is 
at the first position.

\begin{definition}
An inversion of type 1  of $\pi$ is a pair $(b_1,B_j)$, 
where $b_1\in B_i$ for some $i$ ($1\leqslant i<j$) and $b_1>c_1$ for some $c_1\in B_j$.
An inversion of type 2 of $\pi$ is a pair $(b_2, B_j)$, 
where $b_2\in B_i$ for some $i$ ($0\leqslant i<j$) and $b_2>c_2$ for some $c_2\in B_j$ and $b_1\not\in B_j$, where 
$a_i$ means  integer $a$ with  subscript $i=1,2$. Let $\inv_i(\pi)$ be the number of inversions of $\pi$ of type $i=1,2$ and 
set
$\inv(\pi)=\inv_2(\pi)-\inv_1(\pi).
$
Let $\Pi(n,k,i)$ denote the set of  Jacobi-Stirling  $k$-partitions of $[n]_2$
such that the zero-block contains $i$ numbers with subscript1.
\end{definition}

For example, 
for the Jacobi-Stirling 2-partitions of $[3]_2$ we have the following tableau\\
\begin{center}
\begin{tabular}{|c|c|c|c|}
\hline
JS 2-partitions of $[3]_2$&$inv_1$&$\inv_2$&$\inv$\\
\hline
$\{\}_0,\; \{1_1, 1_2, 3_2\}, \; \{2_1, 2_2, 3_1\}$&$0$&$0$&$0$\\
\hline
$\{\}_0,\; \{1_1,1_2, 3_1\}, \; \{2_1,2_2, 3_2\}$&$1$&$0$&$-1$\\
\hline
$\{3_2\}_0,\; \{1_1,1_2, 3_1\}, \; \{2_1,2_2\}$&$1$&$1$&$0$\\
\hline
$\{ 3_2\}_0,\; \{1_1,1_2\}, \; \{2_1,2_2, 3_1\}$&$0$&$1$&$1$\\
\hline
$\{2_2\}_0,\; \{1_1,1_2, 2_1\}, \; \{3_1,3_2\}$&$0$&$0$&$0$\\
\hline
$\{2_1\}_0,\; \{1_1,1_2,2_2\}, \; \{3_1,3_2\}$&$0$&$0$&$0$\\
\hline
$\{3_1\}_0,\; \{1_1,1_2, 3_2\}, \; \{2_1,2_2\}$&$0$&$1$&$1$\\
\hline
$\{3_1\}_0,\; \{1_1,1_2\}, \; \{2_1,2_2, 3_2\}$&$0$&$0$&$0$\\
\hline
\end{tabular}
\end{center}
Thus,
$$
\sum_{\pi\in \Pi(3,2,0)}q^{\inv(\pi)}=3+q+q^{-1}\quad \text{and}\quad 
\sum_{\pi\in \Pi(3,2,1)}q^{\inv(\pi)}=2+q.
$$
\begin{thm} \label{thm2} For any  positive integers  $n$ and $k$ and 
$0\leqslant i\leqslant n-k$ 
we have 
$$
a_{n,k}^{(i)}(q)=\sum_{\pi\in \Pi(n,k,i)}q^{\inv(\pi)}.
$$

 \end{thm}
\begin{proof} We  prove by induction on $n\geqslant 1$. 
The identity is clearly true for $n=1$.  Assume $n>1$. We divide  $\Pi(n,k,i)$ into  three  subsets as follows.
\begin{itemize}
\item $\{n_1, n_2\}$ forms a single block, the enumerative  polynomial is $a_{n-1,k-1}^{(i)}(q)$.
\item $n_1$ is in the zero-block and $n_2$ is in a non zero-block, the enumerative  polynomial is \linebreak $(1+q+\cdots +q^{k-1})a_{n-1,k}^{(i-1)}(q)$.
\item $n_1$ is in a non zero-block $B_j$, so $(n_1, B_i)$ is an inversion of type 1 for any $i=j+1, \ldots, k$, 
and $n_2$ is in any other block $B_l$ ($l\neq j$), so $(n_2, B_i)$ is an inversion for any $i=l+1, \ldots, k$ and $l\neq j$, so the enumerative  polynomial is 
\[
\biggl(\sum_{j=1}^k   q^{-(k-j)}\biggl( \sum_{l=0}^{j-1}q^{k-l-1}+\sum_{l=j+1}^{k}q^{k-l} \biggr)\biggr)a_{n-1,k}^{(i)}(q)=[k]_q[k]_{q^{-1}}a_{n-1,k}^{(i)}(q).
\]
\end{itemize}
Summing up, we have 
$$
a_{n,k}^{(i)}(q)=a_{n-1,k-1}^{(i)}(q)+[k]_qa_{n-1,k}^{(i-1)}(q)+[k]_q[k]_{q^{-1}} a_{n-1,k}^{(i)}(q).
$$
This is equivalent to \eqref{mod q Stirling 2nd}.
\end{proof}

For  a permutation $\sigma$ of  $[n]_0:=[n]\cup \{0\}$ (resp. $[n]$) and
 for $j\in [n]_0$ (resp. $[n]$), denote by \linebreak
$\Orb_\sigma(j)=\{ \sigma^\ell(j): \ell \geqslant 1 \}$ the orbit of $j$ and $\min(\sigma)$
  the set of  its positive cyclic minima, i.e.,
  $$
\min(\sigma)=\{ j \in [n]: j=\min(\Orb_\sigma(j)\cap [n]) \}.
  $$

\begin{definition}  Given a word $w=w(1)\ldots w(\ell)$ on the finite alphabet $[n]$,
a  letter $w(j)$ is a \emph{ record} of $w$ if $w(k)>w(j)$ for every $k\in \{1,\ldots, j-1\}$.
We define $\rec(w)$  to be the number of records of $w$ and $\rec_0(w)=\rec(w)-1$.
\end{definition}

For example, if $w={\bf 5}7{\bf  4}86{\bf 2}3{\bf 1}9$, then
the records are  $5, 4, 2,1$. Hence $\rec(w)=4$.

\begin{definition}
Let $\mathcal{P}(n,k, i)$ be the set of all  pairs $(\sigma,\tau)$ such that
$\sigma$ is a permutation of $[n]_0$, $\tau$  is a permutation of  $[n]$, and both have $k$ cycles.
\begin{itemize}
 \item[i)] 1 and 0 are in the same cycle in $\sigma$;
\item[ii)] among their nonzero entries, $\sigma$ and $\tau$ have  the same cycle minima;
\item[iii)] $\rec_0(w)=i$, where $w=\sigma(0)\sigma^2(0)\ldots \sigma^l(0)$ with $\sigma^{l+1}(0)=0$.
\end{itemize}
\end{definition}

As Foata and Han~\cite{FoataHan2009}, we  define the $B$-code of a permutation $\sigma$ of $[n]$ 
based on the decomposition of each permutation as a product of disjoint cycles. For a permutation
 $\sigma=\sigma(1)\sigma(2)\cdots \sigma(n)$ and each $i=1,2,\ldots, n$ let $k:=k(i)$ be the smallest integer $k\geqslant 1$ such that $\sigma^{-k}(i)\leqslant i$. 
 Then, define 
 $$
 \text{B-code} (\sigma)=(b_1, b_2, \ldots, b_n)\quad \text{with}\quad b_i:=\sigma^{-k(i)}(i)\quad  (1\leqslant i\leqslant n).
 $$
We define the  sorting index for permutation $\sigma$ of $[n]$  by 
$
\sor(\sigma)=\sum_{i=1}^n (i-b_i),
$
while for a  permutation $\sigma$ of $[n]_0$ we define the modified sorting index by
$
\sor_0(\sigma)=\sum_{i=1}^n (i-b_i'),
$
where $b_i'=b_i$ if $\sigma^{-1}(i)\neq 0$ and $b_i'=i$ if  $\sigma^{-1}(i)=0$. 
Finally, for any pair $(\sigma, \tau)$ in $\mathcal{P}(n,k, i)$ we define the statistic 
$$
\sor(\sigma, \tau)=\sor(\tau)-\sor_0(\sigma).
$$
\begin{thm} \label{thm3} We have 
$ \displaystyle \ 
b_{n,k}^{(i)}(q)=\sum_{(\sigma,\tau)\in \mathcal{P}(n,k,i)}q^{\sor(\sigma,\tau)}.
$
\end{thm}

\begin{proof}
 We proceed  by induction on $n\geqslant 1$.  The case $n=1$ is clear. Assume that $n>1$. We divide $\mathcal{P}(n,k,i)$ into three parts:
\begin{itemize}
\item[(i)] the pairs $(\sigma,\tau)$ such that $\sigma^{-1}(n)=n$.
Then $n$ forms a cycle in both $\sigma$ and $\tau$ and the enumerative polynomial 
is clearly $b_{n-1,k-1}^{(i)}(q)$.
\item[(ii)]  the pairs $(\sigma,\tau)$ such that $\sigma^{-1}(n)=0$.
We can construct such pairs starting from a pair $(\sigma',\tau')$ in \linebreak${\mathcal{P}(n-1,k, i-1)}$  as follows:
we insert $n$ in $\sigma'$ as image of $0$, i.e., $\sigma(0)=n$,  $\sigma(n)=\sigma'(0)$ and 
$\sigma(i)=\sigma'(i)$ for $i\neq 0, n$, and then  we insert  $n$  in $\tau'$) by choosing 
an $j\in [n]$ and define $\tau(j)=n$, $\tau(n)=\tau(j)$ and $\tau(l)=\tau'(l)$ for $l\neq j, n$.
Clearly, the enumerative polynomial is  $[n-1]_qb_{n-1,k}^{(i-1)}(q)$.
\item[(iii)] the  pairs $(\sigma,\tau)$ such that $\sigma^{-1}(n)\not\in \{0,n\}$.
We can construct such pairs by first choosing an ordered pair $(\sigma',\tau')$ in $\mathcal{P}(n-1,k,i)$ and
then inserting $n$ in $\sigma'$ and $\tau'$, respectively.
Clearly, the corresponding enumerative polynomial is  given by $ [n-1]_q[n-1]_{q^{-1}}b_{n-1,k}^{(i)}(q)$.
\end{itemize}
Summing up,  we get the following equation:
  \begin{equation} \label{eqgnk}
  b_{n,k}^{(i)} (q)= b_{n-1,k-1}^{(i)}(q) + [n-1]_q b_{n-1,k}^{(i-1)}(q) + [n-1]_q[n-1]_{q^{-1}} b_{n-1,k}^{(i)}(q).
  \end{equation}
  By \eqref{qjs1} and \eqref{eq:defb},  it is easy to see that  the coefficients $b_{n,k}^{(i)}(q)$ satisfy the same recurrence.
   \end{proof}

For example, for $n=3$ and $k=2$  the pairs $(\sigma, \tau)$ and the associated statistics are as follows:
\begin{align*}
\begin{tabular}{|c|c|c|c|c|c|c|}
\hline
$(\sigma,\tau)$&$\rec_0(\sigma)$&$B_0$-code $\sigma$&$B$-code $\tau$& $\sor(\tau)$&$\sor_0(\sigma)$&$\sor(\sigma, \tau)$\\ 
\hline
$(0\,1)(2\,3), \; (1)(23)$&0&(1,2,2)&(1,2,2)&1&1&0\\ 
\hline
$(0\,1)(2\,3), \; (13)(2)$&0&(1,2,2)&(1,2,1)&2&1&1\\ 
\hline
$(0\,1\,2)\,(3), \; (1\,2)(3)$&0&(1,1,3)&(1,1,3)&1&1&0\\ 
\hline
$(0\,1\, 3)\,(2), \; (13)(2)$&0&(1,2,1)&(1,2,1)&2&2&0\\ 
\hline
$(0\,1\, 3))(2), \; (1)(23)$&0&(1,2,1)&(1,2,2)&1&2&-1\\ 
\hline
$(0\,3\, 1)(2), \; (1)(23)$&1&(0,2,3)&(1,2,2)&1&1&0\\ 
\hline
$(0\,3\,1)(2), \; (1)(23)$&1&(0,2,3)&(1,1,3)&1&1&0\\ 
\hline
$(0\,2\,1)(3), \; (1)(23)$&1&(0,2,3)&(1,2,1)&2&1&1\\ 
\hline
\end{tabular}
\end{align*}
Thus,
$$
\sum_{(\sigma,\tau)\in \mathcal{P}(3,2,0)}q^{\sor(\sigma,\tau)}=3+q+q^{-1},\quad \sum_{(\sigma,\tau)\in \mathcal{P}(3,2,1)}q^{\sor(\sigma,\tau)}=2+q.
$$
\subsection{A symmetric generalisation  of Jacobi-Stirling numbers} \label{subsec: generalisation}
As suggested by Richard Askey (private communication in 2013), 
it is natural to 
consider the pair of connection coefficients $\{(S_{z,w}(n,k), s_{z,w}(n,k))\}_{n\geqslant k\geqslant 0}$ satisfying 
\begin{align*}
x^{n}=\sum_{k=0}^n S_{z,w}(n,k)& \prod_{i=0}^{k-1} (x-(i+z)(i+w)),\\
\prod_{i=0}^{n-1} (x-(i+z)(i+w))&=\sum_{k=0}^n s_{z,w}(n,k)x^{k}.
\end{align*}
It is readily seen that we have the following recurrence relation for $0\leqslant k\leqslant n$
\begin{align*}
S_{z,w}(n+1,k+1)&=S_{z,w}(n,k)+(z+k+1)(w+k+1)S_{z,w}(n,k+1),\\
s_{z,w}(n+1,k+1)&=s_{z,w}(n,k)-(z+n)(w+n)s_{z,w}(n,k+1),
\end{align*}
with $S_{z,w}({n},{k}) = s_{z,w}(n,k)= 0 ,\ \text{ if } k \notin \{1,\ldots, n\},$ and $S_{z,w}({0},{0}) = s_{z,w}(0,0)=1$.

It is clear that $S_{z,w}(n,k)$ are  symmetric polynomials in $z$ and $w$ with non negative integer coefficients and 
$$
\sum_{n\geqslant k}S_{z,w}(n,k)x^n=\prod_{i=1}^k \frac{x}{1-(i+z)(i+w)x}.
$$
When $w=0$ these numbers reduce to Jacobi-Stirling numbers.
Mimicking the arguments in  \cite{ZengGelineau}, we can prove similar results for the symmetric variants.  
Hence, we will juste state the  combinatorial interpretations of these numbers and  omit the proofs.
\begin{definition}
A double signed $k$-partition of $[n]_2=\{1_1, 1_2, \ldots, n_1, n_2\}$ is a partition  of  $[n]_2$ into 
 $k+2$ subsets  $(B_0, B_0', B_1, \ldots, B_k)$   such that 
\begin{enumerate}
\item there are two distinguishable
zero blocks $B_0$ and $B_0'$, any of which may be empty;
\item there are $k$ indistinguishable nonzero blocks, all nonempty, each of which
contains both copies of its smallest element and does not contain both copies of any other number;
\item each  zero block  does not contain both copies of any number and $B_0'$ may contain only numbers with subscript $2$.
\end{enumerate}
\end{definition}

Let $\Pi(n,k)$ be the set of double signed $k$-partitions of $[n]_2$.
For $\pi\in \Pi(n,k)$ denote by $s(\pi)$ (resp. $t(\pi)$) the number of integers  with subscript 1 (reps. 2) in $B_0$ (reps. $B_0'$) 
of $\pi$.

\begin{thm}\label{Thm: Stirling in z and w}
The polynomial $S_{z,w}(n,k)$ is the enumerative polynomial of $\Pi(n,k)$
with $z$ enumerating the numbers with subscript 1 in $B_0$ and $w$ enumerating the numbers with subscript 2 in $B_0'$, i.e., 
$$
S_{z,w}(n,k)=\sum_{\pi \in  \Pi(n,k)}z^{s(\pi)}w^{t(\pi)}.
$$
%
\end{thm}

For example, all the double signed $k$-partitions of $[2]_2$ ($1\leqslant k\leqslant 2$) with the corresponding weight are 
\begin{itemize}
\item[$k=0$:] $\pi=\{\{1_1, 2_1\}_0, \{1_2,2_2\}_0'\}$;  with weight $z^2w^2$;
\item[$k=1$:] 
$\pi_1=\{\{2_2\}_0, \{\}_0', \{1_1, 1_2,2_1\}\}$ with weight  $1$;\\
$\pi_2=\{\{2_1\}_0, \{\}_0', \{1_1, 1_2, 2_2\}\}$ with weight $z$;\\
$\pi_3=\{\}_0, \{2_2\}_0', \{1_1,1_2,2_1\}\}$ with weight $w$;\\
$\pi_4=\{\{1_1\}_0,  \{1_2\}_0', \{2_1,2_2\}\}$ with weight $zw$;\\
$\pi_5=\{\{2_1\}_0, \{2_2\}_0', \{1_1,1_2\}\}$ with weight $zw$.
\item[$k=2$:] $\pi=\{\{\}_0, \{\}_0',  \{1_1,1_2\}, \{2_1,2_2\}\}$ with weight $1$.
\end{itemize}
Thus,  by the Theorem \ref{Thm: Stirling in z and w}, we have
\begin{align*}
x^2&=S_{z,w}(2,0)+S_{z,w}(2,1) (x-zw)+S_{z,w}(2,2)(x-zw)(x-(z+1)(w+1)),
\end{align*}
where 
$
S_{z,w}(2,0)=z^2w^2,\quad
S_{z,w}(2,1)=1+z+w+2zw,\quad 
S_{z,w}(2,2)=1.
$
%
\begin{definition}
Let $\mathcal{P}_0(n,k)$ be the set of all  pairs $(\sigma,\tau)$ of permutations of 
 $[n]_0$ such that
$\sigma$  and $\tau$  both have $k$ cycles and 
\begin{itemize}
 \item[i)] 1 and 0 are in the same cycle in $\sigma$ and $\tau$.
\item[ii)] Among their nonzero entries, $\sigma$ and $\tau$ have  the same set of cycle minima.
\end{itemize}
Let  $\rec_0(w_\sigma)$ be the number of left-to right minima of the word  $w_\sigma=\sigma(0)\sigma^2(0)\ldots \sigma^l(0)$ with $\sigma^{l+1}(0)=0$.
\end{definition}
\begin{thm}\label{sym1}
The polynomial $(-1)^{n-k}s_{z,w}(n,k)$ is the enumerative polynomial of pairs $(\sigma, \tau)$ in 
$\mathcal{P}_0(n,k)$ with $z$ enumerating the ${\rec_0(w_\sigma)}$ and $w$ enumerating ${\rec_0(w_\tau)}$, i.e.,
$$
(-1)^{n-k}s_{z,w}(n,k)=\sum_{(\sigma, \tau)\in \mathcal{P}_0(n,k)}z^{\rec_0(w_\sigma)}w^{\rec_0(w_\tau)}.
$$
\end{thm}

Let $(a)_n=a(a+1)\ldots (a+n-1)$ for $n\geqslant 1$ and $(a)_0=1$. The Wilson polynomials (see \cite{Koekoek}) are more natural and simply expanded in the basis of polynomials in $x^2$ given by $\{(a+\imath x)_n(a-\imath x)_n\}_{n\geqslant0}$, where $\imath^2=-1$. So, it is legitimate to refer to the pair of set of numbers $\{(w(n,k),W(n,k))\}_{0\leqslant k\leqslant n}$ performing the following change of basis 
\begin{align}
(a+\imath x)_n&(a-\imath x)_n=\sum_{k=0}^n w(n,k) x^{2k},\\
x^{2n}&=\sum_{k=0}^n W(n,k)(a+\imath x)_k(a-\imath x)_k,
\end{align}
as the \emph{Wilson numbers of first and second kind}.
Then, by Newton's interpolation formula we have 
 \begin{align}
w(n,k)&=e_{n-k}(a^2,(a+1)^2,\ldots, (a+n-1)^2),\\
W(n,k)&=\sum_{r=0}^k\frac{(-1)^{n-r}(a+r)^{2n}}{r!(k-r)!(2a+r)_r(2a+2r+1)_{k-r}},
\end{align}
where $e_k(x_1, \ldots, x_n)$ denotes the $k$-th elementary symmetric polynomial of $x_1, \ldots, x_n$.
%
It is readily seen that 
\begin{align}
W(n+1,k+1)&=W(n,k)-(a+k+1)^2W(n,k+1),\\
w(n+1,k+1)&=w(n,k)+(a+n)^2w(n,k+1),
\end{align}
with 
$W({n},{k}) = w({n},{k})= 0 ,\ \text{ if } k \notin \{1,\ldots, n\},$ and $W({0},{0}) = w({0},{0})=1$, $n\geqslant 0$.

It is easy to derive from Theorems \ref{Thm: Stirling in z and w}  and  \ref{sym1} that
\begin{align}
(-1)^{n-k}W(n,k)&=\sum_{\pi \in \Pi(n,k)}a^{s(\pi)+t(\pi)},\\
s_{z,w}(n,k)&=\sum_{(\sigma, \tau)\in \mathcal{P}_0(n,k)}a^{\rec_0(w_\sigma)+\rec_0(w_\tau)}.
\end{align}

\begin{remark} It is worth to notice that the change of basis between $\{(a+\imath x)_n(a-\imath x)_n\}_{n\geqslant0}$ and \linebreak$\{ (x^2 + (a-1)^2)^{n}\}_{n\geqslant0}$ is performed by the aforementioned pair of Jacobi-Stirling numbers \linebreak $\{\big( (-1)^{n+k} \js_{n+1}^{k+1}(2\alpha), \JS_{n+1}^{k+1}(2\alpha) \big)\}_{0\leqslant k\leqslant n}$, see \cite[Remark 3.9]{LouYak13}. Indeed, the latter pair of numbers is associated to the linear change of basis in the vector space of polynomials in the variable $x^2$ performed by the Kontorovich-Lebedev integral transform after a slight modification of the kernel, as described in  \cite{LouYak13}. 

\end{remark}

Of course, we can formally work out a $q$-version of the numbers $S_{z,w}(n,k)$ and  $s_{z,w}(n,k)$
by mimicking our approach to the $q$-Jacobi-Stirling numbers. However,  from analytical point of view, 
an appropriate $q$-version of Wilson numbers should be the 
connection coefficients between the basis $\{x^{n}\}_{n\geqslant 0}$ and the Askey-Wilson basis $\{ (az,a/z;q)_n\}_{n\geqslant 0}$
with  $x=(z+1/z)/2$ \cite[Chap. 15 and 16]{IsmailBook} and 
 $(a,b;q)_n=(a;q)_n(b;q)_n$. Thus, the Askey-Wilson numbers of the second and first kind may be defined by
\begin{align}
 x^{n}&=\sum_{k=0}^n W_q(n,k)(az,a/z;q)_k,\label{AW2}\\
 (az,a/z;q)_n&=\sum_{k=0}^n w_q(n,k) x^{k},\label{AW1}
 \end{align}
where $x=(z+1/z)/2$. 
By Newton's formula we have 
\begin{align}
W_q(n,k)=\frac{1}{2^n}\sum_{j=0}^k q^{k-j^2}a^{-2j} \frac{(q^{j}a+q^{-j}/a)^n}{(q,q^{-2j+1}/a^2; q)_{j} (q,q^{2j+1}a^2;q)_{k-j}}.\label{aw2}
\end{align}
Since  $(az,a/z;q)_n=\prod_{k=0}^{n-1}(1-2axq^k+a^2q^{2k})$, 
we derive from  \eqref{AW1} that
\begin{align}
w_q(n,k)=(-1)^k (2a)^n q^{n\choose 2} e_{n-k}\biggl(\frac{a^{-1}+a}{2}, \frac{(aq)^{-1}+aq}{2}, \ldots, 
\frac{(aq^{n-1})^{-1}+aq^{n-1}}{2}\biggr).
\end{align}
The following recurrence relations follow easily from \eqref{AW2} and \eqref{AW1}: 
\begin{align}
W_q(n+1,k+1)&=-\frac{1}{2aq^k}W_q(n,k)+\frac{1+a^2q^{2k}}{2aq^k}W_q(n,k+1),\\
w_q(n+1,k+1)&=-2aq^n w_q(n,k)+(1+a^2q^{2n})w_q(n,k+1).
\end{align}
Let $T_q(n,k)=(-2a)^kq^{k\choose 2}W_q(n,k)$. Then $T_q(n,k)$ are polynomials in $a$ and $q$ with nonnegative coefficients
satisfying the recurrence
\begin{align}
T_q(n+1, k+1)=T_q(n,k)+(1+a^2q^{2k}) T_q(n,k+1).
\end{align}
Though these numbers should be  interesting because they are
related to the \emph{Askey-Wilson operator},  
their  combinatorics seem to be  quite different from that of Wilson numbers or $q$-Jacobi-Stirling numbers.  We plan  to 
investigate  this topic in a future work. 
 
 
 \section{Even order $q$-differential equations fulfilled by the $q$-classical polynomials}\label{Sec: qDiff eq}

To come up with the explicit expressions of $\mathcal{L}_{k;q}$ given in Theorem \ref{Thm: 2k th order diff eq} we do not make use of the analytical properties of the measures of each of the $q$-classical sequences, nor of the properties of the basic hypergeometric corresponding series. (These two approaches would require to split the analysis for each of the $q$-classical sequences). Rather, we base our study on the algebraic approach, within the principles developed in \cite{MaroniTheorieAlg} and in the survey \cite{KherijiMar2002}. Here, the information on the orthogonality measures is hidden in the corresponding linear functionals. This methodology has the merit of gathering all the properties in coherent and consistent framework amongst all the possible $q$-classical cases. The $q$-classical sequences and the corresponding $q$-classical operators will be treated as whole. Based on the properties of the corresponding eigenvalues, there will be essentially two cases to single out: depending on whether the degree of $\Phi$ is less than or equal to 2. Whenever the $q$-classical linear functional admits an integral representation via a weight function, we are able to write $\mathcal{L}_{k;q}$ in a $q$-self adjoint manner, as in Corollary \ref{Cor: Lkq via Uq}. Therefore, we start with preliminary results on the $q$-classical sequences to the sequel.

\subsection{Background}\label{sec: preliminaries}
 
The orthogonality measures will be represented by a moment linear functional that satisfies certain properties. The moment linear functionals (also called {\it forms}, see for instance \cite{MaroniTheorieAlg}) are elements of $\mathcal{P}'$, the dual space of the vector space of polynomials $\mathcal{P}$. We adopt the standard notation of $\langle u,f\rangle$ to consider the action of $u\in\mathcal{P}'$ over $f\in\mathcal{P}$. The sequence of moments of $u\in\mathcal{P}'$ results from its action on the canonical sequence $\{x^{n}\}_{n\geqslant 0}$ and will be denoted as $(u)_{n}:=\langle u,x^{n}\rangle, n\geqslant 0 $.  Any linear operator $T:\mathcal{P}\rightarrow\mathcal{P}$  
has a transpose $^{t}T:\mathcal{P}'\rightarrow\mathcal{P}'$ defined by 
$
	\langle{}^{t}T(u),f\rangle=\langle u,T(f)\rangle\,,\quad u\in\mathcal{P}',\: f\in\mathcal{P}.
$
For example, for any nonzero complex number $a$, any linear functional $u$ and any polynomial $g$, let $ D_qu$ and $gu$ be the forms defined as usual by  
$\langle g u , f \rangle:= \langle u, g f \rangle \ , \ 
	\langle h_{a} u , f \rangle:= \langle u, h_{a} f \rangle = \langle u, f(ax) \rangle $ and 
\begin{align} 
	\label{Transpose Hq} 
	\langle D_{q}u, f \rangle := -\langle u, D_{q}f \rangle   \ \text{ for any }\ f\in\mathcal{P}. 
\end{align}
Thus, with some abuse of notation, $D_q$ acting on linear functionals is minus the transpose of $D_q$ on polynomials: ${}^{t}D_{q}:= - D_{q} $ (see \cite{KherijiMar2002, MaroniTheorieAlg} for more details). For  any $f,g\in\mathcal{P}$ and $u\in\mathcal{P}'$, the following properties hold: 
\begin{align}
& \label{prod fg}
	D_{q}\left( f g\right)(x) 
	= f(qx) \left(D_{q}g\right)(x) + \left(D_{q}f \right)(x)  g(x),
\\
	& \label{prod fu}
	D_{q}\left( f(x) u\right)
	= f(q^{-1}x) \ \left(D_{q}u\right) + q^{-1}(D_{q^{-1}}f )(x)  \ u.
\end{align}
%
Such properties can be generalised, so that from \eqref{prod fg} we obtain a $q$-Leibniz formula \cite{Hahn1949} for the product of two polynomials 
\begin{equation}\label{qLeibniz f g}
	D_{q}^{n}\left( f g\right)(x) 
	=\sum_{k=0}^{n} \qbinom{n}{k}{q} \big(D_{q}^{n-k}f\big)(q^{k}x) \ \big(D_{q}^{k}g\big)(x)
	\ , \ n\geqslant 0, 
\end{equation}
whilst, from \eqref{prod fu} we deduce the $n$th order $q$-derivative of the product of a polynomial $f$ by a  linear functional  $u$: 
\begin{equation}\label{qLeibniz f u}
	D_{q}^{n}\left( f u\right)
	=\sum_{k=0}^{n} \qbinom{n}{k}{q^{-1}} q^{-(n-k)} \big(D_{q^{-1}}^{n-k}f\big)(q^{-k}x) \ \big(D_{q}^{k}u\big)
	\ , \ n\geqslant 0. 
\end{equation}

The dual sequence $\{u_{n}\}_{n\geqslant 0}$ of a given MPS $\{P_{n}\}_{n\geqslant 0}$ is a subset of $\mathcal{P}'$ whose elements are uniquely defined by $
	\langle u_{n},P_{k}  \rangle := \delta_{n,k}, \; n,k\geqslant 0,
$  
where $\delta_{n,k}$ represents the {\it Kronecker delta} symbol. Any element $u$ of $\mathcal{P}'$ can be written in a series of any dual sequence $\{ u_{n}\}_{n\geqslant 0}$ of a MPS  $\{P_{n}\}_{n\geqslant 0}$ \cite{MaroniTheorieAlg}: 
\begin{equation} \label{u in terms of un}
	u = \sum_{n\geqslant 0} \langle u , P_{n} \rangle \;{u}_{n} \; .
\end{equation}

Given a MPS $\{P_{n}\}_{n\geqslant0}$, let us consider the MPS of its $k$th $q$-derivatives, $\{P_{n}^{[k]}\}_{n\geqslant0}$, defined by 
\begin{equation}\label{def Pn k}
	P_{n}^{[k]}(x) := \frac{1}{ [n+1;q]_{k}} D_{q}^{k} P_{n+k}(x) \ , \ n\geqslant0.
\end{equation}
This definition readily implies a relation between the corresponding dual sequences. Precisely, for any integers $j$ and $k$, with $0\leqslant j\leqslant k$, the dual sequence $\{u_{n}^{[k]}\}_{n\geqslant0}$ of $\{P_{n}^{[k]}\}_{\geqslant0}$  and the dual sequence $\{u_{n}^{[k-j]}\}_{n\geqslant0}$ of $\{P_{n}^{[k-j]}\}_{\geqslant0}$ are related by 
\begin{equation}\label{rel between dual seq un k and un k-j}
	D_{q}^{j}(u_{n}^{[k]}) = (-1)^{j} [n+1;q]_{j}\,u_{n+j}^{[k-j]} 
	\ , \ n\geqslant 0. 
\end{equation}
%
%

Whenever a form $u\in\mathcal{P}'$ is such that there exists a PS  $\{P_{n}\}_{n\geqslant0}$ so that $\langle u , P_{n} P_{m} \rangle = k_{n} \delta_{n,m}$ with $k_{n}\neq0$, for $n,m\geqslant 0$, then $u$ is called a {\it regular} form  \cite{MaroniTheorieAlg,MaroniVariations}. The PS $\{P_{n}\}_{n\geqslant 0}$ is then said to be orthogonal with respect to $u$ and we can assume the system (of orthogonal polynomials) to be monic  and the original form $u$ is proportional to $u_{0}$. There is a unique MOPS $\{P_{n}(x)\}_{n\geqslant 0}$ with respect to the regular form $u_{0}$ and it can be characterised by the (popular) second order recurrence relation (see, for instance, \cite{ChiharaBook})
\begin{align} \label{MOPS rec rel} 
&	\left\{ \begin{array}{@{}l}
		P_{0}(x)=1\,; \quad P_{1}(x)= x-\beta_{0}, \vspace{0.15cm}\\
		P_{n+2}(x) = (x-\beta_{n+1})P_{n+1}(x) - 
		\gamma_{n+1} \, P_{n}(x)\ ,\ n\geqslant 0,  
	\end{array} \right. 
\end{align}
where $\beta_{n}=\frac{\langle u_{0},x P_{n}^2  \rangle}{\langle u_{0}, P_{n}^2  \rangle}$ and $\gamma_{n+1}=\frac{\langle u_{0}, P_{n+1}^2  \rangle}{\langle u_{0}, P_{n}^2  \rangle}$ for all $n\geqslant 0$. In this case, the elements of the corresponding dual sequence $\{u_{n}\}_{n\geqslant 0}$ are given by 
\begin{equation}\label{dual seq of MOPS}
	u_{n}=\left(\langle u_{0},P_{n}^{2}\rangle\right)^{-1} P_{n} u_{0}
	\ , \ n\geqslant 0. 
\end{equation}
When $u\in\mathcal{P}'$ is regular, let $\Phi$ be a polynomial such
that $\Phi\, u=0$, then $\Phi=0$     \cite{MaroniTheorieAlg,MaroniVariations}.

The $q$-classical polynomials are orthogonal sequences (see Definition \ref{def q cla}) that share a number of properties. Among them, we recall: 
\begin{prop}\label{Prop qclassical}\cite{KherijiMar2002} Let $\{P_{n}\}_{n\geqslant 0}$ be an MOPS and $u_{0}$ be the corresponding regular form. The following statements are equivalent:
\begin{enumerate}
\item[(a)] $\{P_{n}\}_{n\geqslant 0}$ is a $q$-classical sequence. 
\item[(b)] There are two polynomials $\Phi$ and $\Psi$ with $\Phi$ monic, $\deg \Phi\leqslant 2$ and $\deg\Psi=1$ such that $u_{0}$ fulfills 
\begin{equation}\label{u0 Dq classical}
	D_{q}\bigl( \Phi u_{0}\bigr) + \Psi u_{0}=0  .
\end{equation}
\item[(c)]  There exists a sequence of nonzero numbers $\{\chi_n\}_{n\geqslant 1}$, and two polynomials $\Phi$ and $\Psi$ with $\deg \Phi\leqslant 2$, $\deg\Psi=1$ and $\Phi$ monic such that $\{P_{n}\}_{n\geqslant 0}$ fulfils the $q$-differential equation \eqref{Lq def}. 
\item[(d)] There is a monic polynomial $\Phi$ with $\deg \Phi\leqslant 2$ and a sequence of nonzero numbers $\{ \vartheta_{n}\}_{n\geqslant 0}$ such that 
\begin{equation}\label{Rodrigues}
	P_{n} u_{0} =  \vartheta_{n} D_{q}^{n} \biggl( 
			\biggl(\prod_{\sigma=0}^{n-1} q^{-\sigma\deg\Phi} \Phi(q^{\sigma} x)\biggr)\  u_{0}\biggr)
	\ , \ n\geqslant0. 
\end{equation}
\end{enumerate}
\end{prop} 

Whenever $\{P_{n}\}_{n\geqslant 0}$ is ${q}$-classical, then so is $\{P_{n}^{[k]}\}_{n\geqslant 0}$, for any positive integer $k$ and we have:

\begin{cor}\label{cor: Pn k Dq classical} \cite{KherijiMar2002} Let $k$ be a positive integer. If $\{P_{n}\}_{n\geqslant 0}$ is ${q}$-classical, then so is $\{P_{n}^{[k]}\}_{n\geqslant 0}$. The corresponding ${q}$-classical form $u_{0}^{[k]}$ fulfills 
$ \quad
	D_{q}\biggl( \Phi_{k} u_{0}^{[k]}\biggr) + \Psi_{k} u_{0}^{[k]} = 0 \quad 
$
with 
$\quad 
	\Phi_{k}(x)= q^{-k\deg\Phi} \Phi(q^{k} x)$,    	
$\quad 	\Psi_{k}(x) = q^{-k\deg\Phi} \biggl( \Psi(x) - [k]_{q} \ \bigl(D_{q^{k}} \Phi\bigr)(x)\biggr)   
$ 
and it is related to $u_{0}$ via 
$\ 	u_{0}^{[k]}
	= \zeta_{k} \biggl( \prod_{\sigma=0}^{k-1}  \Phi_\sigma(x)  \biggr) u_{0},  
\ $
where $\zeta_{k}$ is such that $(u_{0}^{[k]})_{0}=1$, with  $\zeta_{0}=1$. 
Moreover,  the sequence $\{P_{n}^{[k]}\}_{n\geqslant 0}$  fulfills 
\begin{equation} \label{qBochner eq for kth qderivatives}
	\Phi_{k}(x) D_{q}\circ D_{q^{-1}} \left(P_{n}^{[k]}(x) \right) 
	- \Psi_{k}(x)D_{q^{-1}} \left(P_{n}^{[k]}(x) \right) 
	= \chi_{n}^{[k]} P_{n}^{[k]}(x) , 
\end{equation}
where  
$$
	\chi_{n}^{[k]} =  [n]_{q^{-1}}\left( [n-1]_{q} \tfrac{\Phi_{k}''(0)}{2} - \Psi_{k}'(0) \right) 
= \left\{ \begin{array}{ll@{}l@{}}  
	 - [n]_{q^{-1}} q^{-(\deg \Phi )k}\ \Psi'(0) 
	 	& \text{if }  \deg\Phi=0,1, \vspace{0.2cm}\\
	~\frac{ [n]_{q^{-1}}}{q^{2k+1}}  \Big(  [n+2k]_{q} 
		-\left( 1 + q    \Psi'(0)\right) \Big) 
		& \text{if }  \deg\Phi=2, & \text{for }n\geqslant 0.
	\end{array}\right. 
$$
\end{cor}


Observe that the nonzero numbers  $\vartheta_{n}$  in the Rodrigues type formula \eqref{Rodrigues} are related to the coefficients $\zeta_k$ in Corollary \ref{cor: Pn k Dq classical} via (see \cite[p.65]{KherijiMar2002} for further details): 
$$
	\vartheta_{n} \quad = \quad \frac{ (-1)^{n}}{[n]_{q}!} <u_{0},P_{n}^{2}> \zeta_{n}  
		\quad = \quad q^{-\frac{n(n-1)}{2} \deg\Phi} \  [n]_q! \left( \prod_{\sigma=1}^n \chi_\sigma^{[n-\sigma] }
	\right)^{-1},  \ n\geqslant 0.
$$
In addition, we also note that if $\deg \Phi=2$, then $
	\chi_{n-k}^{[k]} = q^{k-1}\Big( 
		[n]_{q^{-1}} (z+[n]_q) - [k]_{q^{-1}} (z+[k]_q)
	\Big)
$, for any $0\leqslant k\leqslant n$, where $ \ z= -(1 +q\Psi'(0)) $.

\subsection{Even order $q$-differential equations fulfilled by $q$-classical polynomials}

In this section we prove Theorem \ref{Thm: 2k th order diff eq}, where we will mainly follow an algebraic approach. Afterwards, we explain the cases where the even order $q$-differential operator $\mathcal{L}_{k;q}$ admits a representation in a $q$-self adjoint version, as stated in Corollary \ref{Cor: Lkq via Uq}.

\begin{proof}[{\bf Proof of Theorem \ref{Thm: 2k th order diff eq}}]
In the light of  Corollary \ref{cor: Pn k Dq classical}, the $q$-classical character of the sequence $\{P_{n}\}_{n\geqslant 0}$ implies the $q$-classical character of $\{P_{n}^{[m]}\}_{n\geqslant 0}$, no matter the choice of the nonnegative integer $m$. This encloses the orthogonality of each of the sequences $\{P_{n}^{[m]}\}_{n\geqslant 0}$ with respect to $u_0^{[m]}$ and therefore the elements of the corresponding dual sequences, say   $\{u_{n}^{[m]}\}_{n\geqslant 0}$,  can be written as 
$u_n^{[m]} = \left( < u_{0}^{[m]} ,  \left(P_{n}^{[m]}\right)^{2}>\right)^{-1}  P_{n}^{[m]} u_0^{[m]}$ for any integer $n\geqslant 0$. Consequently, the relation  \eqref{rel between dual seq un k and un k-j} becomes
\begin{equation}\label{rel Pnu0 j and Pnu0 k}
	D_{q}^{\nu}(P_{n}^{[k]}u_{0}^{[k]}) = 
		\varpi_{n}(k,k-\nu) 
		P_{n+\nu}^{[k-\nu]}u_{0}^{[k-\nu]} 
	\ , \ n\geqslant 0,
\end{equation}
where 
\begin{equation}\label{coef omega n k k-j}
	\varpi_{n}(k,k-\nu)= (-1)^{\nu}  [n+1;q]_{\nu}
				\frac{< u_{0}^{[k]} ,  \left(P_{n}^{[k]}\right)^{2}>}
				{ < u_{0}^{[k-\nu]} ,  \left(P_{n+\nu}^{[k-\nu]}\right)^{2}> } 
				\ , \ n\geqslant 0. 
\end{equation}

In this case, the corresponding regular forms $u_{0}^{[\nu]}$, with $\nu=0,1,\ldots,k $, satisfy the $q$-differential equation 
\begin{equation}\label{q classical functional eq nu}
	D_q\left(  \Phi_{\nu}(x) \ u_{0}^{[\nu]} \right) + \Psi_{\nu}(x) \ u_{0}^{[\nu]} =0, 
\end{equation} where $\Phi_{\nu}$ and $\Psi_{\nu}$ are given in Corollary \ref{cor: Pn k Dq classical}, respectively. By recalling \eqref{prod fu}, then, based on the fact that $u_{0}^{[k-1]} $ satisfy \eqref{q classical functional eq nu} with $\nu=k-1$, we  successively obtain
\begin{align*}
	& < \Phi_{k-1}(x) \ u_{0}^{[k-1]} ,  x^mP_{n}^{[k]} > 
	\quad = \quad  \tfrac{-1}{[n+1]_q}  	<D_q\left(  x^m\ \Phi_{k-1}(x) \ u_{0}^{[k-1]} \right) ,   P_{n+1}^{[k-1]} > \\
	& \quad =  \quad  \tfrac{-1}{[n+1]_q}  \ 	< u_{0}^{[k-1]} 
				,  \Big(- \Psi_{k-1}(x)  q^{-m}x^m 
	 +  q^{-1} \Phi_{k-1}(x)[m]_{q^{-1}}x^{m-1}  \Big)  P_{n+1}^{[k-1]} >
	 , \ m,n\geqslant 0. 
\end{align*}
By virtue of the orthogonality of  $\{P_{n}^{[k-1]}\}_{n\geqslant 0}$ with respect to $u_{0}^{[k-1]}$ and because  of the degree of the polynomials $\Phi_{k-1}$ and $\Psi_{k-1}$, we deduce 
$$
	\begin{array}{@{}l@{}}
		< \Phi_{k-1}(x) \ u_{0}^{[k-1]} ,  x^m P_{n}^{[k]} > \\
		\qquad= \left\{\begin{array}{lcl}
			0 ,& m\leqslant n-1 , \  \ n\geqslant 1, \\
			- \frac{q^{-n} }{[n+1]_q}    \left( [n]_q \frac{\Phi_{k-1}''(0)}{2} - \Psi_{k-1}'(0) \right) \ <   u_{0}^{[k-1]} , ( P_{n+1}^{[k-1]})^2 >  
			,& m=n,  \ n\geqslant 0, 
		\end{array}\right. 
\end{array}
$$
Since  $\ \Phi_{\nu-1}(x) u_0^{[\nu-1]} = \frac{\zeta_{\nu-1}}{\zeta_\nu} u_0^{[\nu]} \   $, for $\nu\geqslant 1$, where $\zeta_\nu$ is such that $(u_0^{[\nu]})_0=1$, then we can write 
$$
	\frac{<   u_{0}^{[k]} ,  (P_{n}^{[k]})^2 > }{<   u_{0}^{[k-1]} , ( P_{n+1}^{[k-1]})^2 > }
	= - \frac{\zeta_k}{\zeta_{k-1}} \ \frac{ q^{-n}}{ [n+1]_{q}} \left( [n]_q \dfrac{\Phi_{k-1}''(0)}{2} - \Psi_{k-1}'(0) \right), \quad n\geqslant 0, 
$$
which, by finite induction, leads to 
$$
	\frac{<   u_{0}^{[k]} ,  (P_{n}^{[k]})^2 > }{<   u_{0}^{[k-\nu]} ,  (P_{n+\nu}^{[k-\nu]})^2 > }
	= (-1)^{\nu} \frac{\zeta_k}{  \zeta_{k-\nu}} \left( \prod_{\sigma=1}^\nu \frac{q^{-(n+\sigma-1)}}{[n+\sigma]_{q}}
	 \left( [n+\sigma-1]_q \dfrac{\Phi_{k-\sigma}''(0)}{2} - \Psi_{k-\sigma}'(0) \right) 	   \right) , \ n\geqslant 0.
$$
Therefore, the coefficients  $\varpi_{n}(k,\nu)$ in \eqref{coef omega n k k-j} can also be expressed as follows
\begin{equation} \label{omega n k k-nu via lambda}
	\varpi_{n}(k,k-\nu) 
	=  \frac{\zeta_k}{  \zeta_{k-\nu}} \left( \prod_{\sigma=1}^\nu 
	q^{-(n+\sigma-1)} \left( [n+\sigma-1]_q \dfrac{\Phi_{k-\sigma}''(0)}{2} - \Psi_{k-\sigma}'(0) \right)	  \right) 
	 \ , \ n\geqslant 0,
\end{equation}
so that from Equation \eqref{rel Pnu0 j and Pnu0 k}  we obtain
$$
	D_{q}^{\nu}\Big(  
		\ P_{n}^{[k]} \ u_{0}^{[k]}  \Big) 
	=  \frac{\zeta_k}{  \zeta_{k-\nu}} 
		\left( \prod_{\sigma=1}^\nu 
	q^{-(n+\sigma-1)} \left( [n+\sigma-1]_q \dfrac{\Phi_{k-\sigma}''(0)}{2} - \Psi_{k-\sigma}'(0) \right)	  \right) 
		P_{n+\nu}^{[k-\nu]} u_{0}^{[k-\nu]} 
	\ , \ n\geqslant 0. 
$$
Substituting $u_0^{[k-\nu]}$ on the right hand side by its expression given in Corollary \ref{cor: Pn k Dq classical}, we obtain 
$$
	D_{q}^{\nu}\left(  
		\ P_{n}^{[k]} \ u_{0}^{[k]}  \right) 
	=  {\zeta_k} 
		\left( \prod_{\sigma=1}^\nu 
	q^{-(n+\sigma-1)} \left( [n+\sigma-1]_q \dfrac{\Phi_{k-\sigma}''(0)}{2} - \Psi_{k-\sigma}'(0) \right)	  \right) 
	 \ \Big(\prod_{\sigma=0}^{k-\nu-1} \Phi_{\sigma}\Big)
		P_{n+\nu}^{[k-\nu]} (x)
		\ u_{0},
$$
which can be simplified into 
\begin{equation}\label{D nu of Dk of Pn+k u0 k}
	D_{q}^{\nu}\Big(  
		\ \big(D_{q}^{k} P_{n+k}\big)(x) \ u_{0}^{[k]}  \Big) 
	=  {\zeta_k} 
		\left( \prod_{\sigma=1}^\nu 
	\chi_{n+\sigma }^{[k-\sigma]}	  \right) 
	 \ \left(\prod_{\sigma=0}^{k-\nu-1} \Phi_{\sigma}(x)\right)
		\big(D_{q}^{k-\nu} P_{n+k}\big)(x)
		\ u_{0}, \ n\geqslant 0, 
\end{equation}
if we take into consideration the definition of $\{P_{n}^{[\mu]}\}_{n\geqslant 0}$ given in \eqref{def Pn k} alongside with the expression for $\chi_{n }^{[k]}$ given in Corollary \ref{cor: Pn k Dq classical}. In particular, we take   $n=0$ in \eqref{D nu of Dk of Pn+k u0 k} to get 
\begin{equation}\label{D nu of u0 k} 
	D_{q}^{\nu}\left(  
	 u_{0}^{[k]}  \right) 
	=  {\zeta_k} 
		\left( \prod_{\sigma=1}^\nu 
	\chi_{\sigma }^{[k-\sigma]}	  \right) 
	 \ \left(\prod_{\sigma=0}^{k-\nu-1} \Phi_{\sigma}\right)
		P_{\nu}^{[k-\nu]} 
		\ u_{0}
	\ , \quad n\geqslant 0.
\end{equation}
Using the functional version of the $q$-Leibniz rule \eqref{qLeibniz f u}, the left-hand-side of  \eqref{D nu of Dk of Pn+k u0 k} can be rewritten as  
$$
	D_{q}^{k}\left( D_{q}^{k} P_{n+k}(x)\  u_{0}^{[k]}  \right)
	=\sum_{\nu=0}^{k} \qbinom{k}{\nu}{q^{-1}} q^{-(k-\nu)} 
	\Big(D_{q^{-1}}^{k-\nu} \circ D_{q}^{k}\ P_{n+k} \Big)(q^{-\nu}x) \ \big(D_{q}^{\nu}u_{0}^{[k]} \big)
	\ , \ n\geqslant 0,
$$
which, after \eqref{D nu of u0 k}, can be expressed in terms on $u_0$:  
\begin{eqnarray*}
	D_{q}^{k}\left( D_{q}^{k} P_{n+k}(x)\  u_{0}^{[k]}  \right) 
	&=&\sum_{\nu=0}^{k} \qbinom{k}{\nu}{q^{-1}} q^{-(k-\nu)}  
	 {\zeta_k} 
		\Big( \prod_{\sigma=1}^\nu 
	\chi_{\sigma }^{[k-\sigma]}	  \Big) 
	  \Big(\prod_{\sigma=0}^{k-\nu-1} \Phi_{\sigma}(x)\Big) \\
	&&	\times \ P_{\nu}^{[k-\nu]}(x) 
		\Big(D_{q^{-1}}^{k-\nu} \circ D_{q}^{k}\ P_{n+k} \Big)(q^{-\nu}x) \ u_{0}
	\ , \ n\geqslant 0.
\end{eqnarray*}
Bringing this expression into \eqref{D nu of Dk of Pn+k u0 k}, with $\nu$ replaced by $k$, we obtain
\begin{align*}
	&\sum_{\nu=0}^{k} \qbinom{k}{\nu}{q^{-1}} q^{-(k-\nu)}  \ 
		\Big( \prod_{\sigma=1}^\nu 
	 \chi_{\sigma}^{[ k-\sigma]}    \Big) 
	 \ \Big(\prod_{\sigma=0}^{k-\nu-1} \Phi_{\sigma}(x)\Big)
		P_{\nu}^{[k-\nu]} \ 
		\Big(D_{q^{-1}}^{k-\nu} \circ D_{q}^{k}\ P_{n+k} \Big)(q^{-\nu}x) \ u_{0} \\
	& = 
		\Big( \prod_{\sigma=1}^k 
	 \chi_{n+\sigma}^{[k-\sigma]}	  \Big) 
		  P_{n+k}(x) 
		\ u_{0}
	\ , \ n\geqslant 0. 
\end{align*}
According to the regularity of $u_0$, it follows that $y(x)=  P_{n+k}(x) $ satisfies the $q$-differential equation 
$$\sum_{\nu=0}^{k} \Lambda_{k,\nu}(x;q) 
		\bigg(D_{q^{-1}}^{k-\nu} \circ D_{q}^{k}\ y \bigg)(q^{-\nu}x)  
	\ =  \ \Xi_{n}(k;q)\  y(x),
$$
where 
$$
\Lambda_{k,\nu}(x;q)  
	= \qbinom{k}{\nu}{q^{-1}} q^{-(k-\nu)}  \ 
		\Big( \prod_{\sigma=1}^\nu 
	 \chi_{\sigma}^{[ k-\sigma]} 	  \Big) 
	 \ \Big(\prod_{\sigma=0}^{k-\nu-1}q^{-\sigma\deg\Phi } \Phi(q^{\sigma}x)\Big)
		P_{\nu}^{[k-\nu]}(x)
$$
 for $ \nu=0,1,\ldots , k$, and 
$$
	 \Xi_{n}(k;q) =  \prod_{\sigma=0}^{k-1}  \chi_{n-\sigma}^{[\sigma]}
 	  \ , \ n\geqslant 0, 
$$
with $\chi_{n}^{[k]}$ given in Corollary \ref{cor: Pn k Dq classical}. Hence, we get the result if we observe that  
$
P_{\nu}^{[k-\nu]}(x) := \frac{1}{ [\nu+1;q]_{k-\nu}} D_{q}^{k-\nu} P_{k}(x),  
$
and that 
$
	 \qbinom{k}{\nu}{q^{-1}} \frac{q^{-(k-\nu)}}{ [\nu+1;q]_{k-\nu}}
	 	 = \frac{q^{-(k-\nu)(\nu+1)}}{[k-\nu]_q!}, 
$  for  any integers $k,\nu$ such that $0\leqslant \nu\leqslant k$. 
\end{proof}

From the proof of Theorem \ref{Thm: 2k th order diff eq}, we readily deduce a generalisation of the Rodrigues-type formula  \eqref{Rodrigues} (which is recovered  if we set $n=k$ in the identity \eqref{even order diff funct eq} given below). 

\begin{cor} For any positive integer $k$, a $q$-classical polynomial sequence $\{P_{n}\}_{n\geqslant 0}$, orthogonal with respect to the $q$-classical linear functional $u_{0}$, fulfils the relation 
\begin{equation}\label{even order diff funct eq}
	D_{q}^{k}\Bigl(  
		\  \Bigl(\prod_{\sigma=0}^{k-1} \Phi_{\sigma}(x)\Bigr) \ \bigl(D_{q}^{k} P_{n}\bigr)(x) \ u_0  \Bigr) 
	=  \Xi_{n} (k;q) P_{n}(x)
		\ u_{0}
	\ , \text{ for }\ n\geqslant 0,
\end{equation}
where $\Phi_\sigma(x)$ are the monic polynomials given in Corollary \ref{cor: Pn k Dq classical} and 
$ \Xi_n (k;q)$ are the eigenvalues given by \eqref{Xi n k expression}.
\end{cor}

\begin{proof} The result is an immediate consequence of the relation \eqref{D nu of Dk of Pn+k u0 k}, after setting $\nu=k$ and using the expression for $u_0^{[k]}$ given in Corollary \ref{cor: Pn k Dq classical}. 
\end{proof}

Despite \eqref{even order diff funct eq} is actually a functional relation, in some cases it can be written as a $q$-differential relation. This occurs whenever $u_0$ admits an integral representation via a weight function, $W_q(x)$ say, as regular as necessary,  so that: 
\begin{equation} \label{int rep of u0}
	\langle u_0 , f \rangle 
	= \int_{-\infty}^{+\infty} f(x) W_{q}(x) \ \text{d}x 
	,\ \text{ for any } \ f\in\mathcal{P}. 
\end{equation}
This weight function $W_q$ is such that the integral on the RHS of the precedent identity must exist for any polynomial $f$, and must represent the regular form $u_0$. For this reason, $W_q$ needs to be continuous at the origin or such that the integral 
$$
	\lim_{\epsilon \rightarrow 0+}	 \int_{\epsilon}^{1} \frac{W_{q}(x) - W_q(-x)}{x} \text{d}x 
$$
exists and, in addition, the fact that $u_0$ fulfils \eqref{u0 Dq classical} implies that $W_q$ must fulfill the $q$-differential equation  \eqref{eq Uq} (a detailed justification can be found in \cite[p.81-82]{KherijiMar2002}). 

Whenever \eqref{int rep of u0} is admissible, an MOPS with respect to $u_0$ may as well be referred to as orthogonal with respect to the weight function $W_q$. In this case,  from \eqref{even order diff funct eq} we deduce a $q$-self-adjoint version of \eqref{2k order diff eq}: 
\begin{equation}\label{gen Rod}
	q^{-k}D_{q^{-1} }^{k}\biggl(  
		\  \Big( \prod_{\sigma=0}^{k-1} \Phi_{\sigma}(x)\Big) W_{q}(x)  \Big(D_{q}^{k} P_{n}\Big)(x)  \biggr) 
	=  \Xi_n (k;q) \  W_{q}(x) \ P_{n}(x) \ , \text{for }  \ n\geqslant 0. 
\end{equation}
Since $\{P_{n}\}_{n\geqslant 0}$ forms a basis of $\mathcal{P}$,  we get Corollary \ref{Cor: Lkq via Uq}. 

In particular, if we consider $n=k$ case in \eqref{gen Rod}, we recover the known Rodrigues' formula (see \cite[p.70]{Koekoek}, \cite{Weber}): 
$$
		  \Xi_k (k;q) \ P_{k}(x)
		  = [1;q]_{k} q^{-k} (W_{q}(x) )^{-1}D_{q^{-1} }^{k}\biggl(  
		\  \Big( \prod_{\sigma=0}^{k-1} \Phi_{\sigma}(x)\Big)  \ W_{q}(x)  \biggr) 
		 \ , \ k\geqslant 0. 
$$

\subsection{Integral powers of the second-order $q$-differential operator}

%

 The elements of any $q$-classical sequence $\{P_{n}\}_{n\geqslant 0}$ are solutions of the even order $q$-differential equation in $y$
\begin{equation}\label{Lq^k}
	\mathcal{L}_q^{k}[y](x) = (\chi_{n})^{k} y(x) \ , \ n\geqslant 0.
\end{equation}

This raises the natural question of obtaining an explicit expression for $\mathcal{L}_q^{k}$ as well as of relating it to the $2k$th order $q$-differential operator $\mathcal{L}_{k;q}$
given in \eqref{2k order diff eq}. As the elements of a $q$-classical polynomial sequence are eigenfunctions of both operators, these problems may be addressed through the comparison between the eigenvalues $(\chi_{n})^k$ given by \eqref{eigenval Bohner chi n} and $\Xi_{n}(k)$ given by \eqref{Xi n k expression}. Observe that $\Psi'(0)\neq [j]_q$ for any integer $j\geqslant 0$, otherwise the $q$-classical form $u_0$ fulfilling \eqref{u0 Dq classical} would have an undetermined sequence of moments, contradicting the assumptions. For this reason, $\Xi_{n}(k;q)\neq 0$ for  $n\geqslant k$.

\begin{proof}[\textbf{Proof of Theorem \ref{Prop: Lk}}]  We begin by relating the $k$th power of the eigenvalues $\chi_n$ given by \eqref{eigenval Bohner chi n} and the eigenvalues $\Xi_{n}(k;q)$  given by \eqref{Xi n k expression} for any integer $n\geqslant 0$. In the light of equalities \eqref{q fact to monomials2} with $x=[n]_{q^{-1}}$, when $\deg\Phi<2$ we have 
$$
	(\chi_{n})^k = \sum_{j=0}^k  {S}_{q^{-1}} (k,j) \,  q^{(\deg\Phi -1)\frac{j(j-1)}{2}} \, (-\Psi'(0))^{k-j} \ 
		\Xi_{n}(j;q) \ , \ n\geqslant0, 
$$
whereas, based on the identity \eqref{qjs2} with $x=[n]_{q^{-1}}\left( z + [n]_{q} \right) $, when $\deg\Phi=2$ we obtain 
$$
	(\chi_{n})^k = \sum_{j=0}^k \JS_{k}^j(z;q^{-1})  q^{\frac{j(j+1)}{2}-k}\, \Xi_{n}(j;q)\ , \ n\geqslant0. 
$$

According to Proposition \ref{Prop qclassical}, the $q$-classical polynomials are solutions of the $2k$-order $q$-differential equation  \eqref{Lq^k} where the operator $\mathcal{L}_q$ and the eigenvalues $\chi_n$ are those given by \eqref{Lq def}.  
On the other hand, Theorem  \ref{Thm: 2k th order diff eq} ensures the $q$-classical polynomials to be solutions of the $2k$-order $q$-differential equation \eqref{2k order diff eq}, whose corresponding eigenvalues $\Xi_{n}(k;q)$ are given in Theorem \ref{Thm: 2k th order diff eq}. 
The result now follows from the fact that any $q$-classical sequence forms a basis of $\mathcal{P}$. 
\end{proof}


In a similar manner, for a given positive integer $k$ we can express the eigenvalues $\Xi_{n}(k;q)$ in terms of $(\chi_{n})^k$ for any integer $n\geqslant 0$. Indeed, after  $q\rightarrow q^{-1}$, \eqref{q fact to monomials1} and  \eqref{qjs1},  with $x= [n]_{q^{-1}}$ and 
$x= [n]_{q^{-1}}  \left(  [n]_{q} -A\right) $, respectively,  imply the following identities  
$$
	\Xi_{n}(k;q)
	= \left\{\begin{array}{lcl}
	\ds 		q^{(1-\deg\Phi)\frac{k(k-1)}{2}}   \sum_{j=0}^{k}  {s}_{q^{-1}} (k,j)  
				\, (-\Psi'(0))^{k-j}\,  (\chi_n) ^j
			& \text{if} & \deg\Phi=0,1, \\
	\ds 		q^{-\frac{k(k+1)}{2}}\,    \sum_{j=0}^{k}  \Jc_{k}^j (z;q^{-1}) (k,j) \, q^{j}\,  
				(\chi_n) ^j
			& \text{if} & \deg\Phi=2, 
	\end{array}\right.
$$
which hold for any integer $n\geqslant 0$, where $z= -(1 +q\Psi'(0)) $.

By reversing \eqref{Lq k and Lqk Rel1}
using the $q$-Stirling or $q$-Jacobi-Stirling numbers  of the first kind (or evoking analogous arguments to those in the proof of Theorem~\ref{Prop: Lk}), 
we deduce  the reciprocal identities of \eqref{Lq k and Lqk Rel1}. More precisely, we obtain Corollary \ref{Cor: reciprocal identities}. 

All the $q$-classical polynomials were studied in \cite{KherijiMar2002}. With the equivalence relation 
$$
	\{P_n(x)\}_{n\geqslant 0} \sim \{R_n(x)\}_{n\geqslant 0} 
	\quad \text{if } \quad 
	\exists\ a \neq 0 \quad \text{such that }\quad R_n(x)=a^{-n} P_n(ax), \ n\geqslant 0, 
$$
several equivalence classes were obtained and these are summarised in Table  \ref{Table qclassical}. Since the $q$-classical polynomials are solutions to the $q$-differential equation \eqref{Lq def} and the polynomial coefficient $\Phi$, of degree 2 at most, plays a fundamental role in this work, we have arranged the information according to its properties. 

{\small \begin{table}[h]\caption{The $q$-classical polynomials. }\label{Table qclassical}
\begin{center}~\ 
\begin{tabular}{|c|c|}
\hline $\deg \Phi $  & {\bf $q$-classical MOPS }\\
\hline \hline
0 & \parbox{10cm}{\vspace{0.2cm}Al-Salam Carlitz polynomials \ $\cdot$ \  Discrete q-Hermite polynomials\vspace{0.2cm}} \\
1 &  \parbox{10cm}{\vspace{0.2cm} Big q-Laguerre \ $\cdot$ \ q-Meixner  \ $\cdot$ \ Wall q-polynomials \\
		 \ q-Laguerre polynomials \ $\cdot$ \    Little q-Laguerre polynomials  \\
		   q-Charlier I polynomials\vspace{0.2cm} }
		 \\
 2 (with double root) & \parbox{10cm}{\vspace{0.2cm}Alternative q-Charlier polynomials  \ $\cdot$ \ Stieltjes-Wigert q-polynomials\vspace{0.2cm}}\\ 
2 (with 2 single roots) & \parbox{10cm}{\vspace{0.2cm}Little q-Jacobi polynomials  \ $\cdot$ \  q-Charlier II polynomials  \\
			  Generalized Stieltjes-Wigert q-polynomials  \ $\cdot$ \ 
			Big q-Jacobi  \\
			 Bi-generalized Stieltjes-Wigert q-polynomials }\\
			 \hline
\end{tabular}
\end{center}
\end{table} }

\section{Particular cases}\label{Sec:examples}
Among all the $q$-classical polynomials, we single out some examples from Table \ref{Table qclassical} that illustrate the results here obtained. These are mere illustrations, since the results apply to any $q$-classical family.  Among other things, it permits to obtain differently many expressions for even order $q$-differential equations fulfilled by the $q$-classical polynomials. 
 
\subsection{Al-Salam-Carlitz polynomials}
The monic   Al-Salam-Carlitz polynomials are 
$$
	P_{n}(x,a;q) 
	= (-a)^n q^{\binom{n}{2}}{}_2\phi_1 \left(\begin{array}{c}q^{-n} , x^{-1}\\ 0 \end{array};q,\frac{qx}{a}\right)
 	= \sum_{k=0}^n {n\brack k}_q  (-a)^{n-k}   q^{\binom{n-k}{2}} (x^{-1};q)_{k} \ x^k
$$
and satisfy a recurrence relation of the type \eqref{MOPS rec rel}  with  $\beta_n=(1+a)q^n$ and $\gamma_{n+1}=  a q^n (q^{n+1}-1) $, 
for $n\geqslant 0$, see \cite[p.534]{Koekoek}, \cite{KherijiMar2002}. They are orthogonal with respect to the regular linear functional $u_0$ fulfilling \eqref{u0 Dq classical} with $\Phi(x)=1$, 
 $\Psi(x)= (a(q-1))^{-1}(x-(1+a))$, provided that $a\neq0$. 
 They are positive definite whenever $a<0$ and $0<q<1$ or when $a>0$ and $q>1$.  The case where   $a=-1$ and $0<q<1$ corresponds to the so-called  {Discrete $q$-Hermite polynomials}. We further notice that this MOPS is also Appell with respect to the $q$-difference operator $D_q$, so is to say that $P_n^{[k]} (x,a;q) = P_{n}(x,a;q)$.

In the light of Theorem \ref{Thm: 2k th order diff eq}, for any positive integer $k$,  the  Al-Salam-Carlitz polynomials are eigenfunctions of the $q$-differential operator 
$$
	\mathcal{L}_{k;q}[y](x) 
	 := \sum_{\nu=0}^{k} \frac{[1;q^{-1}]_k}{[1;q^{-1}]_{n-k}}
		  \frac{q^{-k} }{(a(q^{-1}-1))^{\nu}} 		P_{\nu} (x,a;q)
		\bigg(D_{q^{-1}}^{k-\nu} \circ D_{q}^{k}\ y \bigg)(q^{-\nu}x)  . 
$$
Following Proposition \ref{Prop: Lk}, for each $k\geqslant 1$, we may write
$$
	\mathcal{L}_{q} ^k[y](x)
	= \sum_{j=0}^{k}  {S}_{q^{-1}} (k,j) \,  q^{ -\frac{j(j-1)}{2}} \, (a(q-1))^{-(k-j)}\,  	
				\mathcal{L}_{j;q}[y](x). 
$$

When $0<q<1 $ and $a<0 $,  the regular form $u_0$ admits an integral representation as in \eqref{int rep of u0} via the weight function \cite[p.87]{KherijiMar2002}
$$
	W_q(x) =K (qx ;q)_\infty \ (a^{-1}qx ; q)_\infty \mathbbm{1}_{[aq^{-1} ,q^{-1} ]}(x)
	\quad \text{with } \quad
	K^{-1}= \int_{aq^{-1}}^{q^{-1}}  (qt ; q)_\infty  (a^{-1}qt ; q)_\infty dt ,
$$
where $\mathbbm{1}_{A}(x)$ stands for the characteristic function over the set $A$ (see introduction).
Hence, according to Corollary \ref{Cor: Lkq via Uq},  we may alternatively write 
$$
	\mathcal{L}_{k;q}[y](x) 
	:= \biggl(  q^k (qx ;q)_\infty \ (a^{-1}qx ; q)_\infty \big)^{-1}
		D_{q^{-1}}^{k}\biggl(    (qx ;q)_\infty \ (a^{-1}qx ; q)_\infty \bigl(D_{q}^{k} y(x)\bigr)   \biggr),
$$
which is valid for any $y\in\mathcal{P}$. 
When $k=1$, the latter reduces to the well known $q$-Sturm-Liouville equation fulfilled by the Al Salam-Carlitz $q$-polynomials (see, for instance, \cite[p.472]{IsmailBook}).

\subsection{The Stieltjes-Wigert polynomials}

The monic Stieltjes-Wigert polynomials  are 
\begin{eqnarray*}
	P_{n}(x;q) = (-1)^n q^{-\frac{n}{2}(2n+1)} {}_1\phi_1 \left(\begin{array}{c}q^{-n} \\ 0 \end{array};q,-q^{n+3/2}x\right) = \sum_{k=0}^n \frac{(-1)^{n+k} q^{k(k+\frac{1}{2})-n(n+\frac{1}{2})}}{(q;q)_{n-k}} x^k. 
\end{eqnarray*} 
They are orthogonal with respect to $u_0$ fulfilling \eqref{u0 Dq classical} with 
$\Phi(x)= x^2$ and $\Psi(x)= -  (q-1)^{-1} \{ x-q^{-3/2}\}$, and they satisfy a recurrence relation of the type \eqref{MOPS rec rel}  with  $\beta_n=(1+q -q^{n+1})q^{-2n-3/2} $ and $\gamma_{n+1}=(1-q^{n+1})q^{-4(n+1)} , \ n\geqslant 0  $, see \cite{KherijiMar2002}, \cite[p.544]{Koekoek}.  
When $0<q<1$,  $u_0$ is positive definite, and admits an integral representation via a weight-function $W_q$ as in  \eqref{int rep of u0} which is given by 
$$ 
	W_{q}(x) = \sqrt{\frac{q}{2\pi \ln q^{-1}}} \exp\left( - \frac{\ln^2 x}{2\ln q^{-1}} \right) \mathbbm{1}_{(0,+\infty)}(x).
$$
In the light of Corollary \ref{Cor: Lkq via Uq}, for $0<q<1$, the Stieltjes-Wigert $q$-polynomials are eigenfunctions of 
$$
	\mathcal{L}_{k;q}[y](x)=
		q^{-k}\exp\left( \frac{\ln^2 x}{2\ln q^{-1}} \right)
		D_{q^{-1}}^{k}\Big(  
		\  x^{2k} \ \exp\left( - \frac{\ln^2 x}{2\ln q^{-1}} \right) \bigl(D_{q}^{k} y(x)\bigr)   \Big)  .
$$
According to Theorem \ref{Thm: 2k th order diff eq},  the latter operator can be expanded as follows 
$$
	\mathcal{L}_{k;q}[y](x)= \sum_{\nu=0}^{k}	\qbinom{k}{\nu}{q^{-1}}  
		\alpha_{k,\nu} \ 
		x^{2k-2\nu} P_\nu^{[k-\nu]}(x;q) 
		\bigg(D_{q^{-1}}^{k-\nu} \circ D_{q}^{k}\ y \bigg)(q^{-\nu}x), 
$$
with 
$$\alpha_{k,\nu} 
	= q^{-(k-\nu)}  \ \left([\nu]_q!\right)
		\Big( \prod_{\sigma=1}^\nu 
	 			q^{-2 k+\sigma}  \Big(  [2 k-\sigma]_{q} 
		+\frac{1}{q-1}  \Big)  	  \Big).
$$

Finally, by means of the  $q$-Jacobi-Stirling numbers of second kind, for any positive integer $k$, the $k$th  power of the $q$-differential operator 
$$\mathcal{L}_q:= x^2 D_{q}\circ D_{q^{-1}}  
	+  (q-1)^{-1} \{ x-q^{-3/2}\}   D_{q^{-1} }$$ 
is given by  the formula 
$$
	\mathcal{L}_q^k[y](x) =  \sum_{j=0}^{k}  \JS_{k}^{j}((q-1)^{-1};q^{-1}) \,  q^{\frac{j(j+1)}{2}-k}\,  \mathcal{L}_{j;q}[y](x)
	\ , \ \forall y\in\mathcal{P} , 
$$
which, after \eqref{JSc to Sc}, can be alternatively written as 
$$
	\mathcal{L}_q^k[y](x) =  \sum_{j=0}^{k}  \frac{q^{\frac{j(j-1)}{2}}}{(q-1)^{k-j}}S_q(k,j) \,    \mathcal{L}_{j;q}[y](x)
	\ , \ \forall y\in\mathcal{P} .
$$

\subsection{The  Little $q$-Jacobi polynomials}

The monic Little $q$-Jacobi polynomials  are
\begin{eqnarray*}
	P_{n}(x;a,b|q) &=& \frac{(-1)^n q^{\binom{n}{2}}(aq;q)_{n} }{(abq^{n+1};q)_{n}}  {}_{2}\phi_{1}\left(\begin{array}{c}q^{-n},abq^{n+1}\\ aq \end{array};q,qx\right) \\
		&=&\frac{(aq;q)_{n}}{(abq^{n+1};q)_{n}} 
		\sum_{k=0}^n {n\brack k}_q \frac{(-1)^{n-k}(abq^{n+1};q)_{k}}{(aq;q)_{k}}
		q^{\binom{n-k}{2}} x^k
\end{eqnarray*}
and satisfy the $q$-differential equation  \eqref{Lq def} with 
$\Psi(x)=(abq^2 (q-1))^{-1} \{ (1-abq^2)x + aq -1\}$ and \linebreak$\Phi(x)= x (x-b^{-1}q^{-1})$,  for $n\geqslant 0$, see \cite[p.482] {Koekoek}, \cite{KherijiMar2002}.
Hence,  for each positive integer $k$, we have 
$$
	P_{n}^{[k]}(x;a,b|q) = P_{n}(x;aq^k,bq^k|q) 
		, \quad  n\geqslant 0. 
$$
They form an MOPS with respect to the regular linear functional $u_{0}$ fulfilling \eqref{u0 Dq classical} provided that the parameters $a,b$ satisfy $a,b,ab\neq q^{-(n+2)}$, for $n\geqslant 0$. The linear functional $u_{0}$ is positive definite when one of the following two conditions is realised \cite{KherijiMar2002}
\begin{align*} 
	& 0<q<1 , \ 0<a<q^{-1}, \ b\in]-\infty,q^{-1}[ - \{0\},\\
	& q>1 , \ a > q^{-1}, \ b\in]-\infty,0[ \cup]q^{-1}, +\infty[. 
\end{align*}
We refer to  \cite{KherijiMar2002,Koekoek} for further details, but we highlight some results from \cite[pp. 100-102]{KherijiMar2002}.

For certain values of $a, b$ and $q$, it is possible to get an integral representation for the regular form $u_0$ via a weight function $W_q(x)$, as described in \eqref{int rep of u0}. Setting $a:=q^{\alpha-1}$ with $\alpha>0$, then depending on the range of values for $b$ and $q$, the weight-function $W_q(x)$ can be written as follows: 
\begin{itemize}
\item[-] For $0<q<1$ and $b \in]-\infty,1[ - \{0\}$, we have 
$$
	W_{q}(x) = K x^{\alpha-1} \dfrac{(qx;q)_{\infty }}{(bqx;q)_{\infty}} \mathbbm{1}_{(0,q^{-1})}(x) \quad \text{ with } \ 
	K^{-1}= \int_{0}^{q^{-1}} t^{\alpha-1}\dfrac{(qt;q)_{\infty }}{(bqt;q)_{\infty}} {\rm d}t.
$$ 
\item[-] For $q>1$ and  $b\geqslant 1$, we have
$$
	W_{q}(x) =K x^{\alpha-1} \dfrac{(bx;q^{-1})_{\infty }}{(x;q^{-1})_{\infty}} \mathbbm{1}_{(0,b^{-1})}(x)
	\quad \text{ with } \ 
	K^{-1}=  \int_{0}^{b^{-1}} t^{\alpha-1}\dfrac{(bt;q^{-1})_{\infty }}{(t;q^{-1})_{\infty}} {\rm d}t .
$$ 
\item[-] For $q>1$ and  $b\leqslant 0$, we have
$$
	W_{q}(x) =K  |x|^{\alpha-1} \dfrac{(bx;q^{-1})_{\infty }}{(-|x|;q^{-1})_{\infty}} \mathbbm{1}_{(b^{-1},0)}(x)
	\quad \text{ with } \ 
	K^{-1}=   \int_{0}^{|b|^{-1}} t^{\alpha-1}\dfrac{(|b|t;q^{-1})_{\infty }}{(-t;q^{-1})_{\infty}} {\rm d}t .
$$ 
\end{itemize}

On the other hand, it is possible to represent $u_{0}$ by a discrete measure. Namely, for $0<q<1$, we have  (see also \cite[p.482]{Koekoek})
$$
	u_{0}= \frac{(aq;q)_{\infty}}{(abq^2;q)_{\infty}}
		\sum_{k\geqslant 0} \frac{(bq;q)_{k}}{(q;q)_{k}} (aq)^k \delta_{q^{k}} , 
		\ \text{ with }\   |aq|<1, \  |bq|<1,
$$
whereas, when $q>1$, we can write 
$$
	u_{0}= \frac{(a^{-1}q^{-1};q^{-1})_{\infty}}{(a^{-1}b^{-1}q^{-2};q^{-1})_{\infty}}
		\sum_{k\geqslant 0} \frac{(b^{-1}q^{-1};q^{-1})_{k}}{(q^{-1};q^{-1})_{k}} (a^{-1}q^{-1})^k 
				\delta_{b^{-1}q^{-(k+1)}} , 
				\text{ for }|aq|>1, \  |bq|>1. 
$$
In the light of Corollary \ref{Cor: Lkq via Uq}, whenever the  Little $q$-Jacobi linear functional $u_0$ admits an integral representation via a weight function $W_q(x)$, the  Little $q$-Jacobi polynomials $\{P_{n}(x;a,b|q)\}_{n\geqslant0}$ are eigenfunctions of 
\begin{equation}\label{Lkq Little q Jac}
	\mathcal{L}_{k;q}[y](x)=
		q^{-k}\big(W_{q}(x) \big)^{-1}
		D_{q^{-1}}^{k}\Big(  
		\  \Big(\prod_{\sigma=0}^{k-1} x(x-b^{-1}q^{-(\sigma+1)})\Big) \ W_{q}(x) \bigl(D_{q}^{k} y(x)\bigr)   \Big)  .
\end{equation}
In any of the above mentioned cases, following Theorem \ref{Thm: 2k th order diff eq},  the latter operator can be written as follows 
$$
	\mathcal{L}_{k;q}[y](x)= \sum_{\nu=0}^{k}	\qbinom{k}{\nu}{q^{-1}}  
		\alpha_{k,\nu} (x)
		 \Big(\prod_{\sigma=0}^{k-1} x(x-b^{-1}q^{-(\sigma+1)})\Big) 
		\bigg(D_{q^{-1}}^{k-\nu} \circ D_{q}^{k}\ y \bigg)(q^{-\nu}x)  
$$
with 
$$\alpha_{k,\nu}(x) 
	= q^{-(k-\nu)}  \ 
		\Big( \prod_{\sigma=1}^\nu 
	 [\sigma]_{q^{-1}}q^{-2 k+\sigma}  \Big(  [2 k-\sigma]_{q} 
		+z \Big)  	  \Big)P_{\nu}(x;aq^{k-\nu},bq^{k-\nu}|q),
$$
and  $z= - (1+q \Psi'(0)) 
= \frac{1}{q-1} \left( 1- (abq)^{-1}\right)$.

Finally, by means of the  $q$-Jacobi-Stirling numbers of second kind, for each integer $k\geqslant0$, the $k$th  power of the $q$-differential operator 
\begin{equation}\label{Lq LittleJac}
\mathcal{L}_q:= x (x-b^{-1}q^{-1}) D_{q}\circ D_{q^{-1}}  
	- \left((abq^2 (q-1))^{-1} \{ (1-abq^2)x + aq -1\}\right)   D_{q^{-1} }
\end{equation}
is given by  the formula 
$$
	\mathcal{L}_q^k[y](x) =  \sum_{j=0}^{k}  \JS_{k}^{j}(z;q^{-1}) \,  q^{\frac{j(j+1)}{2}-k}\,  \mathcal{L}_{j;q}[y](x),
$$
for any $y\in\mathcal{P}$, where $z= - (1+q \Psi'(0)) 
= \frac{1}{q-1} \left( 1- (abq)^{-1}\right)$.  According to Corollary \ref{Cor: reciprocal identities} the reciprocal relations can also be established:
$$
	\mathcal{L}_{k;q}[y](x) =  \sum_{j=0}^{k}   \Jc_{k}^j (z;q^{-1}) \,  q^{j-\frac{k(k+1)}{2}}\,  
				\mathcal{L}_{q} ^j[y](x), \ \text{ for any } \ y\in\mathcal{P}. 
$$

\begin{remark} Amongst the $q$-classical sequences it is possible to establish limiting relations between the several $q$-classical families. (For a catalogue of such relations we refer to \cite{Koekoek}.) In the same manner, it is also possible to establish limiting relations between two operators, $\mathcal{L}_{k;q}$ and $\widehat{\mathcal{L}}_{k;q}$ (that have $P_{n}$ and $\widehat{P}_{n}$ as eigenfunctions, respectively), mirroring the limiting relation that the two corresponding $q$-classical sequences $\{P_{n}\}_{n\geqslant 0}$ and $\{\widehat{P}_{n}\}_{n\geqslant 0}$ satisfy.

For example, the monic $q$-Laguerre polynomial sequence $\{\widehat{P}_n(x; \alpha; q)\}_{n\geqslant 0}$ is related to the Little $q$-Jacobi polynomials $\{P_{n}(x;q^{\alpha},b|q)\}_{n\geqslant 0}$, that we have just analysed, via  the limiting relation (see \cite[p.522]{Koekoek}): 
\begin{eqnarray*}
	\widehat{P}_n(x; \alpha; q)
	&=& \lim_{b\to -\infty} \left( (-bq)^n P_{n}(-b^{-1}q^{-1}x;q^{\alpha},b|q) \right) \\
	&=& \frac{(-1)^n (q^{\alpha+1},;q)_n}{q^{n(n+\alpha)}} {}_{1}\phi_{1}\left(\begin{array}{c}q^{-n}\\ q^{\alpha+1} \end{array};q, -q^{n+\alpha+1}x\right), \ n\geqslant 0. 
\end{eqnarray*}
The  operator $\widehat{\mathcal{L}}_{q^{-1}} : =  x D_{q^{-1}} \circ D_q - \frac{q^{\alpha+2}}{q-1}\left(\frac{1}{q^{\alpha+1}} - 1-x \right)D_q$ has the monic $q$-Laguerre polynomials $\widehat{P}_n(x; \alpha; q)$ as eigenfunctions, and, more generally, according to Corollary \ref{Cor: reciprocal identities}, so have the operators
$$
	\widehat{\mathcal{L}}_{k;q^{-1}}[y](x) 
	=  \sum_{j=0}^{k}  {c}_{q^{-1}} (k,j)  \, (-\Psi'(0))^{k-j}\,  
			\widehat{\mathcal{L}}_{q^{-1}}^j[y](x),  
$$
for each $k=1,2,\ldots$. In fact, the operators $\widehat{\mathcal{L}}_{k;q^{-1}}$ are a limiting case of the operators ${\mathcal{L}}_{k;q}$ in \eqref{Lkq Little q Jac} which have the Little $q$-Jacobi polynomials as eigenfunctions. 
Indeed, consider the operator $\widetilde{\mathcal{L}}_q=x(x+1) D_q \circ D_{q^{-1}}- \frac{q^{-\alpha -2} \left(-b (x+1) q^{\alpha +2}+b q+x\right)}{b (q-1)} D_{q^{-1}}$ obtained from  ${\mathcal{L}}_q$ in \eqref{Lq LittleJac} after the change of variable $x\to -b^{-1}q^{-1}x$, that is $\widetilde{\mathcal{L}}_q[y(t)][-b^{-1}q^{-1}x] = {\mathcal{L}}_q[y(-b^{-1}q^{-1}t)][x]  $, so that 
$$
	\widetilde{\mathcal{L}}_q 
	\xrightarrow[b\to -\infty]{}
	x(x+1) D_q \circ D_{q^{-1}}
	- \frac{- (x+1) + q^{-(\alpha+1)}}{q-1}     D_{q^{-1}}
	= q^{-(\alpha+2)} \widehat{\mathcal{L}}_{q^{-1}}.
$$
Besides, recall that $z=\frac{1}{q-1} \left( 1- (q^\alpha bq)^{-1}\right)$ and therefore $\lim\limits_{b\to-\infty} z = \frac{1}{q-1} $, so that, by taking into account the identities \eqref{JSc to Sc}, we have for any $y\in\mathcal{P}$ 
\begin{eqnarray*}
	\lim_{b\to-\infty}\mathcal{L}_{k;q}[y(-b^{-1}q^{-1}t)](x) 
		&=&\lim_{b\to-\infty} \left(\sum_{j=0}^{k}   \Jc_{k}^j (z;q^{-1}) \,  q^{j-\frac{k(k+1)}{2}}\,  
				\widetilde{\mathcal{L}}_{q} ^j[y](x) \right)\\
		&=& q^{\binom{k}{2}-k(\alpha+2)}\widehat{\mathcal{L}}_{k;q^{-1}}[y](x). 
\end{eqnarray*}

\end{remark}

\bigskip

The detailed analysis done for the $q$-classical sequences of Al-Salam-Carlitz polynomials, Stieltjes-Wigert polynomials and Little $q$-Jacobi polynomials, including the limiting relation to the $q$-Laguerre polynomials, can naturally be extended to all the other $q$-classical sequences listed in Table \ref{Table qclassical}. The procedure is entirely similar to the ones just described, as the required properties were already studied by several authors and a synthesis of it can be found in  \cite{KherijiMar2002,Koekoek}. 

In this paper, we have succeeded in providing the explicit expressions for any integral composite power of a second order $q$-differential operator having the $q$-classical polynomials as eigenfunctions. The study required the introduction of the new set of $q$-Jacobi-Stirling numbers, which we have characterised and provided a combinatorial interpretation.  Clearly, by allowing $q\to 1$ we recover all the classical polynomial related results found in \cite{AndrewsEggeGawLittle2013,LittleEvAlJacobi2007,LittleEvAlJacobi2000, ZengGelineau, Loureiro2010,LouMarRocha2006,Miranian} from the viewpoint of both the algebra of these operators and the combinatorial interpretations of the connection coefficients.\\

\noindent{\bf Acknowledgements.}  We thank the referees for careful reading of this manuscript and illuminating comments and remarks.




\end{document}